\RequirePackage{pdf14}
\documentclass[notitlepage,11pt,reqno]{amsart}
\usepackage[foot]{amsaddr}
\usepackage{amssymb,nicefrac,bm,upgreek,mathtools,verbatim,enumerate}
\usepackage[final]{hyperref}
\usepackage[mathscr]{eucal}
\usepackage{dsfont}
\usepackage[normalem]{ulem}
\usepackage{amsopn}

\usepackage[margin=1in]{geometry}
\newcommand{\ttup}[1]{\textup{(}#1\textup{)}}
\newcommand{\stkout}[1]{\ifmmode\text{\sout{\ensuremath{#1}}}\else\sout{#1}\fi}

\newtheorem{lemma}{Lemma}[section]
\newtheorem{theorem}{Theorem}[section]

\newtheorem{corollary}{Corollary}[section]

\theoremstyle{definition}
\newtheorem{definition}{Definition}[section]
\newtheorem{assumption}{Assumption}[section]

\newtheorem{example}{Example}[section]

\theoremstyle{remark}
\newtheorem{remark}{Remark}[section]
\numberwithin{theorem}{section}
\numberwithin{equation}{section}

\hypersetup{
  colorlinks=true,
  citecolor=mblue,
  linkcolor=mblue,
  urlcolor = blue,
  anchorcolor = blue,
  frenchlinks=false,
  pdfborder={0 0 0},
  naturalnames=false,
  hypertexnames=false,
  breaklinks}
\usepackage[capitalize,nameinlink]{cleveref}

\usepackage[abbrev,msc-links,nobysame]{amsrefs} 

\crefname{section}{Section}{Sections}
\crefname{subsection}{Section}{Sections}
\crefname{condition}{Condition}{Conditions}
\crefname{hypothesis}{Hypothesis}{Conditions}
\crefname{assumption}{Assumption}{Assumptions}
\crefname{lemma}{Lemma}{Lemmas}
\crefname{fact}{Fact}{Facts}

\Crefname{figure}{Figure}{Figures}

\crefformat{equation}{\textup{#2(#1)#3}}
\crefrangeformat{equation}{\textup{#3(#1)#4--#5(#2)#6}}
\crefmultiformat{equation}{\textup{#2(#1)#3}}{ and \textup{#2(#1)#3}}
{, \textup{#2(#1)#3}}{, and \textup{#2(#1)#3}}
\crefrangemultiformat{equation}{\textup{#3(#1)#4--#5(#2)#6}}%
{ and \textup{#3(#1)#4--#5(#2)#6}}{, \textup{#3(#1)#4--#5(#2)#6}}%
{, and \textup{#3(#1)#4--#5(#2)#6}}

\Crefformat{equation}{#2Equation~\textup{(#1)}#3}
\Crefrangeformat{equation}{Equations~\textup{#3(#1)#4--#5(#2)#6}}
\Crefmultiformat{equation}{Equations~\textup{#2(#1)#3}}{ and \textup{#2(#1)#3}}
{, \textup{#2(#1)#3}}{, and \textup{#2(#1)#3}}
\Crefrangemultiformat{equation}{Equations~\textup{#3(#1)#4--#5(#2)#6}}%
{ and \textup{#3(#1)#4--#5(#2)#6}}{, \textup{#3(#1)#4--#5(#2)#6}}%
{, and \textup{#3(#1)#4--#5(#2)#6}}

\crefdefaultlabelformat{#2\textup{#1}#3}
\newcommand{\vertiii}[1]{{\left\vert\kern-0.25ex\left\vert\kern-0.25ex\left\vert #1 
    \right\vert\kern-0.25ex\right\vert\kern-0.25ex\right\vert}}
\newcommand{\lamstr}{\lambda^{\mspace{-2mu}*}}
\newcommand{\process}[1]{{\{#1_t\}_{t\ge0}}}

\newcommand{\Uadm}{\mathfrak Z}
\newcommand{\Act}{\mathcal{Z}}
\newcommand{\Usm}{\mathfrak Z_{\mathsf{sm}}}
\newcommand{\bUsm}{\overline{\mathfrak Z}_{\mathsf{sm}}}
\newcommand{\Usms}{\mathfrak Z^*_{\mathsf{sm}}}

\newcommand{\cA}{{\mathcal{A}}}  
\newcommand{\sA}{{\mathscr{A}}}  
\newcommand{\fB}{{\mathfrak{B}}}  
\newcommand{\sB}{{\mathscr{B}}}  
\newcommand{\cC}{{C}}   
\newcommand{\sE}{{\mathscr{E}}} 
\newcommand{\sF}{{\mathfrak{F}}}   
\newcommand{\cG}{{\mathcal{G}}}  
\newcommand{\cI}{{\mathcal{I}}}  
\newcommand{\cIm}{{\widehat{\mathcal{I}}}}
\newcommand{\sJ}{{\mathscr{J}}}  
\newcommand{\fJ}{{\mathfrak{J}}}
\newcommand{\cK}{{\mathcal{K}}}  
\newcommand{\Lg}{{\mathcal{L}}}  
\newcommand{\Lp}{{L}}            
\newcommand{\cT}{{\mathcal{T}}}  
\newcommand{\cN}{{\mathcal{N}}}  
\newcommand{\sP}{{\mathscr{P}}}
\newcommand{\Lyap}{{\mathscr{V}}}  
\newcommand{\cX}{{\mathcal{X}}}

\newcommand{\RR}{\mathds{R}}
\newcommand{\NN}{\mathds{N}}

\newcommand{\Rd}{{\mathds{R}^{d}}}
\DeclareMathOperator{\Exp}{\mathbb{E}}
\DeclareMathOperator{\Prob}{\mathbb{P}}
\newcommand{\D}{\mathrm{d}}
\newcommand{\E}{\mathrm{e}}
\newcommand{\Ind}{\mathds{1}}   

\newcommand{\Sob}{{\mathscr W}}    
\newcommand{\Sobl}{{\mathscr W}_{\mathrm{loc}}} 

\newcommand{\df}{\coloneqq}
\newcommand{\transp}{^{\mathsf{T}}}

\DeclareMathOperator*{\trace}{Tr}
\DeclareMathOperator*{\dist}{dist}
\DeclareMathOperator*{\diam}{diam}

\DeclareMathOperator*{\supp}{support}

\newcommand{\order}{{\mathscr{O}}}
\newcommand{\grad}{\nabla}
\newcommand{\uuptau}{{\Breve\uptau}}

\newcommand{\abs}[1]{\lvert#1\rvert}
\newcommand{\norm}[1]{\lVert#1\rVert}
\newcommand{\babs}[1]{\bigl\lvert#1\bigr\rvert}

\newcommand{\babss}[1]{\biggl\lvert#1\biggr\rvert}
\newcommand{\bnorm}[1]{\bigl\lVert#1\bigr\rVert}

\usepackage{color}
\definecolor{dmagenta}{rgb}{.4,.1,.5}
\definecolor{dblue}{rgb}{.0,.0,.5}
\definecolor{mblue}{rgb}{.0,.0,.7}
\definecolor{ddblue}{rgb}{.0,.0,.4}
\definecolor{dred}{rgb}{.7,.0,.0}
\definecolor{dgreen}{rgb}{.0,.5,.0}
\definecolor{Eeom}{rgb}{.0,.0,.5}

\newcommand{\ttl}{\Large Risk-sensitive control for a class of
diffusions with jumps}
\begin{document}
\title[Risk-sensitive control for a class of diffusions with jumps]
{\ttl}

\author[Ari Arapostathis]{Ari Arapostathis$^{\dag}$}
\address{$^{\dag}$Department of ECE,
The University of Texas at Austin,
EER~7.824, Austin, TX~~78712}
\email{ari@utexas.edu}

\author[Anup Biswas]{Anup Biswas$^\ddag$}
\address{$^\ddag$Department of Mathematics,
Indian Institute of Science Education and Research,
Dr.\ Homi Bhabha Road, Pune 411008, India}
\email{anup@iiserpune.ac.in}

\begin{abstract}
We consider a class of  diffusions controlled through the drift and jump size,
and driven by a jump L\'evy process and a nondegenerate Wiener process,
and we study infinite horizon (ergodic) risk-sensitive control
problems for this model. 
We start with the controlled Dirichlet eigenvalue problem in smooth bounded domains,
which also  allows us to generalize current results in the literature on exit rate
control problems. Then we consider the infinite horizon
average risk-sensitive  minimization and
maximization problems on the whole domain.
Under suitable hypotheses, we establish existence and uniqueness
of a principal eigenfunction for the Hamilton--Jacobi--Bellman (HJB)
operator on the whole space,
and fully characterize stationary Markov optimal controls
as the measurable selectors of this HJB equation.
\end{abstract}


\maketitle


\section{Introduction}
Risk-sensitive control of continuous time processes became popular since the
seminal work of Fleming and McEneaney \cite{FM95}, and evolved rapidly primarily
because of its applications in finance \cite{BP99, FS00}.
Some early literature on the
risk-sensitive control can be found in \cite{DK89,DM97,ADS03,JBE94,J92,Wh81,DL08} and
a book dedicated to this topic is the work of Whittle \cite{Whittle}.
In this article we are interested in the ergodic limit
of the risk-sensitive criterion and
there is a substantial amount of work in the literature on this topic.
See for instance,
\cite{AB18,AB18a,ABS19,Biswas-11,Biswas-11a,Biswas-10,BS18,Nagai96,HS10,Menaldi-05}
and references therein.
The body of work on ergodic risk-sensitive control of general Markov processes
is large and it
is impossible to give a complete list of references.
We cite \cite{BorMey02,DMS99,DMS07,AJ07} for discrete Markov chains and
\cite{GS14,SP15} for continuous time Markov chains.
Though this problem has been studied for the last two decades
some of the important questions
for the problem on the whole space
over an infinite horizon,
like uniqueness of the value function,
verification results etc.,  were addressed only very recently in \cite{AB18,ABS19},
and variational representations are established in \cite{AB18a}.
Risk-sensitive control also attracted immense interest because of its connection
to the study of large deviations of
occupation measures of diffusions \cite{DV75,DVIII,Kaise-06}.
The infinite horizon (ergodic)
risk-sensitive control problem we are addressing can be informally described as follows:
given a controlled stochastic differential equation (with jumps) of the form
\begin{equation}\label{E-sde}
\D X_t \,=\, b_\circ(X_t, Z_t)\, \D t + \upsigma(X_t)\, \D W_t
+ \int_{\RR^m\setminus\{0\}} g(X_{t-}, Z_t, \xi)\, \widetilde\cN(\D t, \D\xi)\,,
\quad X_0=x\in\Rd\,,
\end{equation}
where $b_\circ$ is the drift, $\upsigma$ is the diffusion matrix, $g$ is the `jump-size',
$\widetilde\cN$ is a compensated Poisson process
with a finite characteristic measure $\Pi$,
and $Z$ is an admissible control
taking values in a compact metrizable space $\Act$.

Controlled jump-diffusions with 
a compact action space arise as heavy traffic limits to controlled queueing and
communications networks,
see for instance \cite{Kush00} and references therein.
As shown in \cite{PW09}, for $G/M/n+M$ queues with
asymptotically negligible service interruptions,
the limit queueing process is a one-dimensional L{\'e}vy-driven SDE.
For a description of the controlled dynamics of
these queueing models, we refer the reader to \cite[Section~4.2]{APS19}
and \cite{AHPS-19}.
As can be seen in these papers, the limiting diffusion of these queueing systems
has a compact action space and the driving L{\'e}vy process has a finite characteristic measure,
so it matches the model studied in this paper.
In addition, the jump size does
not depend on the current state or the control parameter.
A similar setting is also used by \cite{MR99,AB-20a} to
study an ergodic control problem.
However, the risk-sensitive control problem for these systems has not been addressed in
the literature, and is open.
\Cref{E-sde} also appears in portfolio optimization problems. 
For instance, if we impose a restriction on 
short selling, then the portfolio wealth process is modeled 
as in \cref{E-sde} where
the portfolio strategies take values in some compact metric space \cite[p.~1555]{DGR18}.

The risk-sensitive control problem aims to minimize
\begin{equation*}
\limsup_{T\to\infty} \,\frac{1}{T}\,
\log\Exp_x\left[\E^{\int_0^T c(X_s, Z_s)\,\D s}\right],
\end{equation*}
over all admissible controls $Z=\{Z_t\}_{t\ge0}$,
where $c$ is a nonnegative running cost function.
We are interested in the characterization of the optimal value,
and the optimal controls.
It is natural to expect that this results in an eigenvalue problem,
namely, an equation of the form
\begin{equation}\label{E-HJB}
\trace (a\grad^2 V)+ \min_{\zeta\in\Act}\,\bigl\{ I[V, x, \zeta] + b(x,\zeta)\cdot \grad V (x)
+ c(x,\zeta) V(x)\bigr\} \,=\, \sE^* V \quad \text{in\ } \Rd\,,
\end{equation}
for some positive $V\in\Sobl^{2,p}(\Rd)$, $p>d$,
where $a\df \frac{1}{2}\upsigma\upsigma\transp$,
\begin{equation*}
b(x, \zeta) \,=\, b_\circ(x, \zeta) - \int_{\RR^m\setminus\{0\}} g(x,\zeta, \xi) \Pi(\D{\xi})\,,
\end{equation*}
and $I[V, x, \zeta]$ denotes
the non-local interaction given by
\begin{equation*}
I[V, x, \zeta] \,=\, \int_{\RR^m\setminus\{0\}}
\bigl(V(x+ g(x,\zeta, \xi))-V(x)\bigr)\Pi(\D{\xi})\,.
\end{equation*}
We refer to $V$ and $\sE^*$ as the \emph{value function} and
\emph{optimal value}, respectively.
Ideally, one expects $\sE^*$ to be the 
principal eigenvalue of the above operator.
However, it is now known from \cite[Example~3.1]{ABS19} that this might not
be the case,
in general, even for continuous controlled diffusions, that is, with $g=0$. 
At the same time, for $g=0$, the above operator has uncountably many generalized
eigenvalues \cite{Berestycki-15}.
In the recent articles \cite{AB18, ABS19}, the 
authors develop a very general set of criteria under which $\sE^*$ coincides
with the generalized principal eigenvalue of the above operator in $\Rd$ (for $g=0$).
It is also shown in \cite{ABS19} that the uniqueness of the principal eigenfunction
is related to its monotonicity with respect to the function $c$.

It is important to mention another class of risk-sensitive control problems known as
Exponential Linear Quadratic Gaussian (ELQG) problems.
These models are quite popular due to their application in mathematical finance,
see for instance \cite{Nagai96,HS10,FS00,DL13} and references therein.
They allow the action set to be unbounded but generally impose a 
more restrictive assumption on the structure of the drift and running cost.
The dynamics are governed by \cref{E-sde}, with $\upsigma$ constant,
$\widetilde\cN\equiv0$, and $b_\circ(x,\zeta)=\zeta$, where $\zeta\in\Rd$.
A typical running cost is
$c(x,\zeta)= h(x) + \frac{\abs{\zeta}^2}{2}$ for some inf-compact function $h$.
More general running costs are studied, but they are essentially perturbations
of this form.

We introduce a risk parameter
$\theta\neq 0$, and define
\begin{equation}\label{E-theta}
\sE^*(\theta)\,\df\,\inf_{Z}\,\limsup_{T\to\infty} \,\frac{1}{\theta T}\,
\log\Exp_x\left[\E^{\int_0^T \theta c(X_s, Z_s)\,\D s}\right].
\end{equation}
Note that $\sE^*(\theta)$ might not be finite for all values of $\theta$.
This is known as the
\emph{breakdown} phenomenon.
It is thus important to determine the range of $\theta$ for which
breakdown does not occur (see \cite{Nagai96,HS10}).
Letting $g=0$ and $W=\frac{1}{\theta}\log V$, we obtain
from \cref{E-HJB} that
\begin{equation}\label{E-HJB1}
\sE^*(\theta) \,=\, \trace \bigl(a\grad^2 W(x)\bigr)
+ H\bigl(x,\grad W(x)\bigr) + h(x) \quad \text{in\ } \Rd\,,
\end{equation}
where
\begin{align*}
H(x,p)&\,\df\,
\theta p\cdot ap +
 \min_{\zeta\in\Act}\,\bigl\{  \zeta\cdot \grad p
+ \frac{\abs{\zeta}^2}{2} \bigr\}\\
&\,=\,  \theta p\cdot ap - \tfrac{1}{2} \abs{p}^2\,,\quad p\in\Rd\,,
\end{align*}
denotes the Hamiltonian.
\Cref{E-HJB1} constitutes an ergodic control problem
and has been studied extensively \cite{BF92,Ichi11,Ichi15}.
Moreover, if $H(x, p)\asymp - |p|^2$, then
the existence and uniqueness of a solution can be studied using the method in \cite{BF92}. 
For ELQG problems this is guaranteed whenever $\theta\leq \delta$ for some small
$\delta>0$ (see \cite{Nagai96}).
On the other hand, when $H(x, p)\asymp |p|^2$, studying existence and uniqueness of
a solution to \cref{E-HJB1} is delicate, and often needs more restrictive
assumptions  (see \cite[Condition~(A3)]{Kaise-06}).
We should also mention the work in \cite{HS10} where the authors 
find the explicit solution of \cref{E-HJB1} for $d=1$ for a particular class of problems. 

Now compare this with the model in the present article that is,
where the action set is compact, but let $g=0$.
In this model the Hamiltonian $H(x,p)$ behaves like $|p|^2$ and the method of \cite{BF92}
does not apply.
For this reason, the authors of \cite{ABS19} studied the  eigen-equation \cref{E-HJB} instead.
It can be shown that no-breakdown is directly related to the finiteness
of the principal eigenvalue of the HJB equation, and this finiteness is assured by \cref{A1.1}
in \cref{S1.2}.
In fact, this assumption can be seen as a variant of \cite[Condition~(A3)]{Kaise-06}.

We mention some general facts about risk-sensitive control.
From the Taylor series expansion of the exponential function, one sees that
the risk-sensitive criterion captures all moments of the
cumulative cost $\int_0^T \theta c(X_s, Z_s)\,\D s$.
Thus, it can be viewed as an extension of the `mean-variance' criterion.
But unlike the latter, it is amenable to dynamic programming.
If we let $\theta\searrow0$ in \cref{E-theta}, we formally recover the
average cost as a limiting case (risk-neutral case).
Also note that in classical criteria, cost minimization is equivalent
to reward maximization by a flip of a sign of the running cost.
In risk-sensitive control, this is not so, you get a different problem
(see \cref{S-max}).

The studies cited above, deal with the case where $g=0$.
In this article we consider the problem where the jump component is present, that
is $g\ne0$.
To the best of our knowledge, there is no work in the literature that considers
ergodic risk-sensitive control problems for jump diffusions.
There are few recent studies \cite{DL11, DL13, DGR18} that consider finite horizon
risk-sensitive control problems for a particular class of
jump diffusions.
The main goals of this paper are the following:
(a) characterize the optimal value $\sE^*$ as the principal eigenvalue of the
quasi-linear operator in \cref{E-HJB},
(b) establish uniqueness of the value function $V$, and
(c) develop verification results for the optimal Markov controls.
We establish all these results under a blanket
geometric ergodicity hypothesis on the dynamics.
Similar results are also obtained for the
risk-sensitive maximization problem without imposing a blanket stability hypothesis,
 but instead, under a near-monotone structural assumption.

We compare the results and methodology in this paper to the existing literature.
There are two main approaches in the study of
ergodic risk-sensitive control problems for the case $g=0$ (with compact action space).
The first approach, consists of formulating a discounted risk-sensitive
control problem, and then, by taking a suitable normalization
of the discounted value function, deriving \cref{E-HJB} as
a vanishing discount limit; see, for instance, \cite{Menaldi-05, Biswas-10}.
In the second approach, one starts from the Dirichlet eigenvalue problem
on bounded domains, and derives \cref{E-HJB} as the limit over a sequence
of expanding domains which covers $\Rd$ \cite{Biswas-11a, ABS19}.
Using the first approach, it is possible
to show that $(V, \sE^*)$ is an eigenpair of \cref{E-HJB},
but concluding that $\sE^*$ is the principal eigenvalue is not an easy task.
This property is important in establishing uniqueness of the value function $V$.
In contrast, the second approach directly obtains $\sE^*$ as the
principal eigenvalue.
It is also important to note that unlike the case of $g=0$,
the logarithmic transformation of the value function $V$ does not lead to a
HJB equation corresponding to some stochastic control problem.
In this work we follow the second approach, taking the path of \cite{ABS19}.
The first hurdle arises from the fact that almost nothing is known for the eigenvalue
problem of the operator in \cref{E-HJB} on bounded domains.
A recent study \cite{DQT} addresses the eigenvalue problem on a bounded domain
for stable-like operators.
However, our operator is not of this type.
So we first study the spectral properties in \cref{S-bounded} for bounded domains.
The next challenge is how to pass to the limit as the domain increases to $\Rd$.
Note that the operator is non-local and Harnack's inequality, which asserts
that the eigenfunctions are locally uniformly bounded, fails, in general, for this
class of operators (see \cite[Example~1.1]{ACPZ18}).
Therefore, the standard method followed in \cite{ABS19,Berestycki-15}
does not apply, and instead,
we use the Lyapunov function to construct a barrier,
together with a novel method to establish a lower bound
of the limit of the eigenvalues on bounded domains (see \cref{L4.2}).
Finiteness of the characteristic measure
$\Pi$ is crucial in this approach. More precisely, we treat the nonlocal integration 
as a zeroth order perturbation of the local pde, and apply a generalized Harnack's inequality.
The same method does not work if $\Pi$ is a non-finite measure.
Our approach also provides an alternative  way of finding the \emph{principal} eigenfunction
and eigenvalue in situations where Harnack's inequality may
not be available, for instance, Markov chains with a general state space.
On the other hand, for the maximization problem, we use the near-monotone property
of the running cost to bound the eigenfunctions.
We should also point out that the analysis of \cite{ABS19} heavily uses the 
\emph{twisted process}, or Doob's $h$-transformation,
but such a transformation is not simple to construct for the jump diffusion model.
So we rely heavily on the stochastic representation of the principal eigenfunction,
and use it cleverly to overcome the difficulties.
To the best of our knowledge,
these are the first results in the literature for
the generalized eigenvalue problem in unbounded domains for semilinear operators
 with a non-local term.
It should also be noted that we do not allow dependence on
the control variable for the diffusion.
This is a standard setting (see \cite{book})
and allows us to construct a strong Markov process under any stationary
Markov control.
To elaborate on this matter, recall that one of our main goals is 
to obtain a verification of optimality result.
The optimal stationary Markov controls $v$ are obtained via a measurable selection argument
(see \cref{T1.1}), hence, we do not have any information about their regularity.
Thus if we allow $\upsigma$ to depend on the control
parameter $\zeta$, then $x\mapsto a\bigl(x,v(x)\bigr)$
is only measurable and not continuous, in general.
Therefore, it becomes difficult to give a meaning
to the solution to \cref{E-sde} under arbitrary stationary Markov controls.
Additional hypotheses are needed to guarantee, for example, that
$x\mapsto a\bigl(x,v(x)\bigr)$ is continuous.
For this reason, we let the diffusion coefficient $\upsigma$ to be independent of the control.

The tools we develop in \cref{S-bounded}
are useful in the study of the exit rate problem, which seeks
to maximize the rate function
\begin{equation*}
\limsup_{T\to\infty}\,\frac{1}{T}\, \log \Prob^Z_x(\uptau>T)\,,
\end{equation*}
where $\uptau$ is the first
exit time from a smooth domain $D$, and $Z$ is an admissible control.
It turns out that the optimal value of this problem
is the principal eigenvalue of a suitable operator.
For $g=0$ the exit rate problem is considered in \cite{BB10, BA15}.
Such problems arise in reliability theory where one often wants to confine the
controlled process to a prescribed region of its state space for as long as possible.
In \cref{S-control}
we provide a complete characterization (see \cref{T3.2}) to the exit rate problem,
including verification of optimality results.

The rest of the paper is organized as follows:
In the next section we describe the model, and state the assumptions and the
main results.
\cref{S-bounded} studies the Dirichlet eigenvalue problem on bounded domains
under more general hypotheses, and \cref{S-control} is devoted to the exit
rate control problem.
The proofs or the results in \cref{S-bounded} are in \cref{S-proofs}.
\cref{S-risk} is devoted to the study of the eigenvalue problem in $\Rd$
and the risk sensitive minimization problem, while \cref{S-max} treats
the maximization problem.

\subsection{Notation.}\label{Snot}
We denote by $\uptau(A)$ the \emph{first exit time} of the process
$\{X_{t}\}$ from the set $A\subset\RR^{d}$, defined by
\begin{equation*}
\uptau(A) \,\df\, \inf\,\{t>0\,\colon X_{t}\not\in A\}\,.
\end{equation*}
The open ball of radius $r$ centered at $x\in\Rd$
is denoted by $\sB_{r}(x)$, and $\sB_{r}$ without an argument denotes
the ball centered at $0$.
We let $\uptau_{r}\df \uptau(\sB_{r})$,
and $\uuptau_{r}\df \uptau(\sB^{c}_{r})$.

The complement and closure
of a set $A\subset\Rd$ are denoted
by $A^{c}$ and $\Bar{A}$, respectively, and  $\Ind_A$
denotes its indicator function.
The minimum (maximum) of two real numbers $a$ and $b$ is denoted by $a\wedge b$ 
($a\vee b$), respectively,
and $a^\pm \df (\pm a)\vee 0$.
The inner product of two vectors $x$ and $y$ in $\Rd$ is denoted
as $x\cdot y$, or $\langle x,y\rangle$, $\abs{\,\cdot\,}$ denotes
the Euclidean norm, $x\transp$ stands for
the transpose of $x$, and
$\trace S$ denotes the trace of a square matrix $S$.

The term \emph{domain} in $\RR^{d}$
refers to a nonempty, connected open subset of the Euclidean space $\RR^{d}$. 
For a domain $D\subset\RR^{d}$,
the space $\cC^{k}(D)$ ($\cC^{k}_{b}(D)$), $k\ge 0$,
refers to the class of all real-valued functions on $D$ whose partial
derivatives up to order $k$ exist and are continuous (and bounded),
$\cC_{\mathrm{c}}^k(D)$ denotes its subset
 consisting of functions that have compact support,
and $\cC_0^k(D)$ the closure of $\cC_{\mathrm{c}}^k(D)$.
The space $\Lp^{p}(D)$, $p\in[1,\infty)$, stands for the Banach space
of (equivalence classes of) measurable functions $f$ satisfying
$\int_{D} \abs{f(x)}^{p}\,\D{x}<\infty$, and $\Lp^{\infty}(D)$ is the
Banach space of functions that are essentially bounded in $D$.
The standard Sobolev space of functions on $D$ whose generalized
derivatives up to order $k$ are in $\Lp^{p}(D)$, equipped with its natural
norm, is denoted by $\Sob^{k,p}(D)$, $k\ge0$, $p\ge1$.
In general, if $\mathcal{X}$ is a space of real-valued functions on $D$,
$\mathcal{X}_{\mathrm{loc}}$ consists of all functions $f$ such that
$f\varphi\in\mathcal{X}$ for every $\varphi\in\cC_{\mathrm{c}}(\mathcal{X})$.
Likewise, we define $\Sobl^{k, p}(D)$.

For a nonnegative function $f\in C(\RR^d)$, 
we use $\order(f)$ to denote the subspace of functions
$g\in C(\RR^d)$ such that
$\sup_{x\in\RR^d} \frac{\abs{g(x)}}{1 + f(x)} < \infty$.

\subsection{Description of the problem}\label{S1.2}
The controlled jump diffusion process $\process{X}$ in $\Rd$ is
governed by the It\^{o} equation
\begin{equation}\label{E1.1}
\D X_t \,=\, b_\circ(X_t,Z_t)\,\D t + \upsigma(X_t)\,\D W_t
+ \int_{\RR^m\setminus\{0\}} g(X_{t-},Z_t,\xi)\, \widetilde\cN(\D t, \D\xi)\,,
\quad X_0=x\in\Rd.
\end{equation}
Here, $W$ is a $d$-dimensional standard Wiener process, and $\widetilde\cN$
is a martingale measure in $\RR^m$, corresponding to a  Poisson
random measure $\cN$.
In other words, $\widetilde\cN(t, A)=\cN(t, A)-t \Uppi(A)$, with
$\Exp[\cN(t, A)]=t \Uppi(A)$ for any Borel subset $A$ in $\RR^m\setminus\{0\}$,
where $\Uppi$ is a finite
measure on $\RR^m\setminus\{0\}$.
The processes $W$ and $\cN$ are independent and defined on a complete
probability space $(\Omega, \sF, \Prob)$.
The control process $\process{Z}$ takes values in a compact metric space $\Act$,
is predictable with respect to $\sF_t$, and is \emph{non-anticipative}:
for $s<t$, $\bigl(W_t-W_s,\, \cN(t, \cdot)-\cN(s, \cdot)\bigr)$ is independent of
\begin{equation*}
\sF_s \,\df\, \text{the\ completion\ of\ }
\sigma\{X_0, Z_r, W_r, \cN(r, \cdot) \,\colon\, r\le s\}
\text{\ relative\ to\ } (\sF, \Prob)\,.
\end{equation*}
The process $Z$ is called an \emph{admissible} control,
and the set of all admissible control is denoted by $\Uadm$. 

Let $a= \frac{1}{2}\upsigma \upsigma\transp$.
We impose the following assumptions to guarantee existence of solution of \cref{E1.1}.

\begin{enumerate}[(A2)]
\item[\hypertarget{A1}{{(A1)}}]
\emph{Local Lipschitz continuity:\/}
for some constant $C_{R}>0$ depending on $R>0$, the functions
$\upsigma\,=\,\bigl[\upsigma^{ij}\bigr]\colon\RR^{d}\to\RR^{d\times d}$,
$b\colon\Rd\times\Act\to\Rd$, and $g\colon\Rd\times\RR^m\to\Rd$ satisfy
\begin{equation}\label{EA2}
\babs{b_\circ(x,\zeta)-b_\circ(y, \zeta)}^2 + \norm{\upsigma(x) - \upsigma(y)}^2
+ \int_{\RR^m\setminus\{0\}} \abs{g(x,\zeta, \xi)-g(y,\zeta, \xi)}^2\,\Uppi(\D\xi)
\,\le\, C_{R}\,\abs{x-y}^2
\end{equation}
for all $x,y\in \sB_R$ and $\zeta\in\Act$, where
$\norm{\upsigma}\df\sqrt{\trace(\upsigma\upsigma\transp)}$.
We also assume that $b$ and $g$ are continuous.

\medskip
\item[\hypertarget{A2}{{(A2)}}]
\emph{Affine growth condition:\/}
For some constant $C_0>0$, we have
\begin{equation}\label{EA3}
\sup_{\zeta\in\Act}\, \bigl\langle b_\circ(x, \zeta),x\bigr\rangle^{+} + \norm{\upsigma(x)}^{2}
+ \sup_{\zeta\in\Act}\, \int_{\RR^m\setminus\{0\}} \abs{g(x,\zeta, \xi)}^2\, \Uppi(\D\xi)
\,\le\,C_0 \bigl(1 + \abs{x}^{2}\bigr) \qquad \forall\, x\in\RR^{d}\,.
\end{equation}

\medskip
\item[\hypertarget{A3}{{(A3)}}]
\emph{Nondegeneracy:\/}
For each $R>0$, it holds that
\begin{equation*}
\sum_{i,j=1}^{d} a^{ij}(x)\eta_{i}\eta_{j}
\,\ge\,C^{-1}_{R} \abs{\eta}^{2} \qquad\forall\, x\in \sB_R\,,\ \forall\,
\eta=(\eta_{1},\dotsc,\eta_{d})\transp\in\Rd\,.
\end{equation*}
\end{enumerate}

With $\fB(\Rd)$ denoting the Borel $\sigma$-algebra of $\Rd$, we define
\begin{equation*}
\nu(x,\zeta,A) \,\df\,
\Uppi\bigl(\{\xi\in\RR^m\setminus\{0\}\,\colon\, g(x,\zeta, \xi)\in A\}\bigr)\,,
\qquad A\in\fB(\Rd)\,,
\end{equation*}
and
\begin{equation}\label{E-barnu}
\Bar\nu \,\df\, \nu(x,\zeta,\Rd) \,=\, \Pi\bigl(\RR^m\setminus\{0\}\bigr)\,.
\end{equation}
Note that
$\int_{\Rd}\abs{z}^2 \, \nu(x,\zeta,\D z)\le C_0 (1+ \abs{x}^2)$ by \hyperlink{A2}{(A2)}.
Also, \hyperlink{A2}{(A2)} and the finiteness of $\Pi$ imply
that $x\mapsto \int_{\Rd}z\, \nu(x,\zeta,\D z)$ has at most affine growth in $x$.
Moreover, since $g$ is continuous, it follows that
$(x, \zeta)\mapsto \int_{\Rd} f(z) \nu(x, \zeta, \D{z})$ is continuous
for every continuous function $f\in\order\bigl(\abs{z}^2\bigr)$.

It is well known that under hypotheses
\hyperlink{A1}{(A1)}--\hyperlink{A2}{(A2)},
the stochastic differential equation in
\cref{E1.1} has a unique strong solution for
every admissible control (see for example, \cite{GS72}).
By a Markov control, we mean an admissible control of the form $v(t,X_t)$
for some Borel measurable function $v\colon \RR_+\times\Rd\to\Act$.
If $v$ is independent of $t$, we call it a stationary Markov control, and the set
of all stationary Markov controls is denoted by $\Usm$.
The hypotheses in \hyperlink{A1}{(A1)}--\hyperlink{A3}{(A3)}
imply the existence of unique
strong solutions under Markov controls, that is, $X_t$ solves
\begin{equation*}
\D X_t \,=\, b_\circ(X_t, v(t, X_{t-}))\, \D t + \upsigma(X_t)\, \D W_t
+ \int_{\RR^m\setminus\{0\}} g(X_{t-}, v(t, X_{t-}), \xi)\, \widetilde\cN(\D t, \D\xi)\,,
\quad X_0=x\in\Rd\,.
\end{equation*}
Indeed, as established in \cite[Theorem~2.8]{Gyongy-96},
using the method of Euler's approximations, the diffusion
\begin{equation}\label{E-sde-aux}
\D{\Tilde X_t} \,=\,
b_\circ\bigl(\Tilde X_t,v(t,\Tilde X_t)\bigr)\,\D{t} -
\biggl(\int_{\RR^m\setminus\{0\}} g(\Tilde X_t,v(t,\Tilde X_t), \xi) \Pi(\D{\xi})\biggr)\D{t}
+ \upsigma(\Tilde X_t)\,\D{W_t}\,,
\end{equation}
with $X_0 = x\in \RR^d$,
has a unique strong solution for any Markov control $v$.
As shown in \cite[Theorem~14]{Skorokhod-89}, since
the L\'evy measure is finite,
the solution of \cref{E-sde} can be constructed in a piecewise fashion
by concatenating the solutions of \cref{E-sde-aux} between consecutive jumps
(see also \cite{Li-03}).
We mention here, parenthetically the work of Veretennikov \cite{Vr80},
which is probably the first one to establish
existence of strong solutions for equations with a bounded measurable drift.

\begin{definition}\label{D1.1}
Let $c\colon\Rd\times\Act\to\RR_+$ be a continuous function
which represents the \emph{running cost}.
Given a control $Z\in\Uadm$, the \emph{risk-sensitive penalty} is defined by
\begin{equation*}
\sE_x(c, Z) \,\df\,
\limsup_{T\to\infty} \, \frac{1}{T}\,
\log \Exp_x\Bigl[\E^{\int_0^T c(X_s, Z_s)\,\D s}\Bigr]\,,
\end{equation*}
and the optimal value is defined as
\begin{equation*}
\sE^* \,\df\, \inf_{x\in\Rd}\, \inf_{Z\in\Uadm}\, \sE_x(c,Z)\,.
\end{equation*}
An admissible control $Z^*$ is called \emph{optimal} if
$\sE_x(c,Z^*)=\inf_{Z\in\Uadm}\, \sE_x(c,Z)$ for all $x\in\Rd$.
\end{definition}

\begin{definition}
We define the semilinear operator $\cI$ by
\begin{equation}\label{E-cI}
\cI f(x) \,\df\, \trace\bigl(a(x)\grad^2 f(x)\bigr) 
+ \inf_{\zeta\in\Act}\,\bigl\{I[f,x,\zeta] + b(x,\zeta)\cdot \grad f(x) + c(x,\zeta) f(x)\bigl\}
\end{equation}
for $f\in \cC^2(\Rd)\cap\cC_b(\Rd)$, where 
\begin{equation}\label{E-I}
\begin{aligned}
I[f,x,\zeta] &\,\df\, \int_{\Rd} \bigl(f(x+z)-f(x)\bigr)\, \nu(x,\zeta,\D{z})\,,\\
b(x,\zeta) &\,\df\, b_\circ(x,\zeta) - \int_\Rd z\,\nu(x,\zeta,\D{z})\,.
\end{aligned}
\end{equation}

We also define the operators $\cA$ and $\cA^c$ mapping
$\cC^2(\Rd)$ to $\cC(\Rd\times\Act)$ by
\begin{equation}\label{EcA}
\begin{aligned}
\cA u(x,\zeta) &\,\df\, \trace(a\grad^2 u(x)) + I[u,x,\zeta] + b(x,\zeta)\cdot\grad u(x)\,,
\\[3pt]
\cA^c u(x,\zeta) &\,\df\, \cA u(x,\zeta) + c(x,\zeta) u(x)\,,
\end{aligned}
\end{equation}
and for $v\in\Usm$, we often use the simplifying notation
\begin{equation}\label{Esimpl}
I_v[f, x]= I[f, x, v(x)]\,,\quad \ b_v(x) \,\df\, b\bigl(x,v(x)\bigr)\,,\quad\text{and\ \ }
c_v(x) \,\df\, c\bigl(x,v(x)\bigr)\,.
\end{equation}
\end{definition}

For many of the results we enforce the following Foster--Lyapunov condition on the dynamics.

\begin{assumption}\label{A1.1}
In (a) and (b) below, $\Lyap \in \cC^2(\Rd)$ is some function taking
values in $[1,\infty)$, $\widehat{C}$ is a positive constant,
and $\cK\subset\Rd$ is a compact set.
\begin{enumerate}
\item[(a)]
If $c$ is bounded, we assume without loss of generality
that $\inf_{\Rd\times\Act}\,c=0$, and that there exists some
constant $\gamma>\norm{c}_\infty$ satisfying
\begin{equation}\label{EA1.1A}
\cA \Lyap(x,\zeta)
\,\le\, \widehat{C} \Ind_{\cK}(x)-\gamma \Lyap(x)\qquad \forall\,(x,\zeta)
\in\Rd\times\Act\,.
\end{equation}

\item[(b)] If $c$ is not bounded, we assume
that there exists an inf-compact function $\ell$
such that $\ell-c$ is inf-compact, and
\begin{equation}\label{EA1.1B}
\cA \Lyap(x,\zeta)
\,\le\, \widehat{C} \Ind_{\cK}(x)-\ell(x) \Lyap(x)
\qquad \forall\,(x,\zeta)\in\Rd\times\Act\,.
\end{equation}
\end{enumerate}
In both cases (a) and (b) we also assume that the map
\begin{equation}\label{EA1.1C}
x\,\mapsto\,\int_{\Rd} \max_{\zeta\in\Act}\,\Lyap\bigl(x+g(x, \zeta, \xi)\bigr)\, \Pi(\D{\xi})
\end{equation}
is locally bounded.
\end{assumption}

As well known (see \cite{ABS19}), if $a$ and $b$ are bounded,
it might not be possible to find an unbounded function $\ell$
satisfying \cref{EA1.1B}.
This is the reason for \cref{EA1.1A}.

Before we proceed further, let us
exhibit two classes of dynamics satisfying \cref{A1.1}.

\begin{example}\label{Ex1.1}
Suppose that
$\sup_{\zeta\in\Act} b(x,\zeta)\cdot x\le - \kappa \abs{x}^\alpha$
outside a compact set
for some $\alpha\in [1, 2]$, and $a$ is bounded. 
Define $\Lyap (x)\,\df\,\exp(\theta \sqrt{\abs{x}^2+1})$.
Then an easy calculation shows that
\begin{equation*}
\begin{aligned}
\cI \Lyap(x) &\,\le\, \kappa_1 \biggl(\Ind_{\cK_1}(x)
+ \theta\frac{1}{\sqrt{\abs{x}^2+1}}
+ \theta^2 \frac{\abs{x}^2}{\abs{x}^2+1}\biggr)\Lyap(x)\\
&\mspace{150mu}- \theta\frac{\abs{x}^{\alpha}}{\sqrt{\abs{x}^2+1}}\Lyap(x)
+ \int_{\Rd} \bigl(\Lyap(x+z)-\Lyap(x)\bigr)\nu(x,\zeta,\D{z})
\end{aligned}
\end{equation*}
for some constant $\kappa_1$, and a compact set $\cK_1$.
Now suppose that $\supp(\nu(x,\zeta,\cdot))\subset B(0, \eta)$ for all $x\in\Rd$
and $\zeta\in\Act$ (i.e., $g$ is bounded).
Then, since
$\Lyap(x+z) \le \Lyap(x)\exp\bigl(2\theta\abs{z}\bigr)$
by the mean-value theorem, we obtain
\begin{equation*}
\int_{\Rd} \bigl(\Lyap(x+z)-\Lyap(x)\bigr)\,\nu(x,\D{z}) \,\le\,
\Lyap(x) (\E^{2\theta\eta}-1)\Bar\nu\,.
\end{equation*}
Thus, if $\alpha>1$,
and we choose $\ell\sim\abs{x}^{\alpha-1}$, \cref{EA1.1B} is satisfied.
For $\alpha=1$, if we assume that $\eta$ is sufficiently small so that 
\begin{equation*}
\int_{\Rd}(\E^{2\theta \abs{z}}-1)\nu(x,\zeta,\D{z}) \,<\,
\theta_1\theta
\end{equation*}
for some $\theta_1<1$ and all $\theta\in(0,1)$,
then by choosing $\theta$ suitably small we obtain \cref{EA1.1A}.
\end{example}

\begin{example}
If the measure $\nu$ is heavy-tailed, it is not possible
to use exponential Lyapunov functions $\Lyap$
like the one used in \cref{Ex1.1}.
Suppose, for simplicity, that $\nu$ is translation invariant, that is, 
$g(x,\zeta,\xi)$ does not depend on $x$ and $\zeta$, and that
$\int \abs{z}^\theta\nu(\D{z})<\infty$ for $\theta\in[0,\theta_\circ)$
for some $\theta_\circ>1$,
and $\int \abs{z}^{\theta_\circ}\nu(\D{z})=\infty$.
In such a case, \cref{E-sde} has a strong solution, even though
\hyperlink{A2}{(A2)} is not satisfied if $\theta_\circ<2$.
With
\begin{equation*}
\fJ[f](x)\,\df\,\int_{\RR^m\setminus\{0\}} \bigl(f(x+z)-f(x)- z\cdot\nabla f(x)\bigr)\,
\nu(\D{z})\,,
\end{equation*}
we write $\cA$ as
\begin{equation}\label{Ex1.2Z}
\cA f(x,\zeta) \,=\, \trace(a\grad^2 f(x)) + \fJ[f](x)
+ b_\circ(x,\zeta)\cdot\grad f(x)\,.
\end{equation}
Note that the drift $b_\circ$ in \cref{E1.1} appears here,
and not the modified $b$ in \cref{E-I}.
Suppose that there exists a positive definite symmetric matrix $S\in\RR^{d\times d}$
such that 
\begin{equation}\label{Ex1.2A}
x\transp S b_\circ(x,\zeta) \,\le\, C_0 - C_1 (x\transp Sx)\,.
\end{equation}
This is, for example, the case for stable linear drifts with a nonlinear
perturbation that has sublinear growth.
Assume also that $\upsigma$ has sublinear growth.
Consider a Lyapunov function $\Lyap\in\cC^2(\Rd)$ which agrees with
$(x\transp Sx)^{\nicefrac{\theta}{2}}$ outside some ball.
Then \cref{Ex1.2A} and the sublinear growth of $\upsigma$ imply that
for every $\epsilon>0$, there exists a constant $C_0'(\epsilon)$ such that
\begin{equation}\label{Ex1.2B}
\trace(a\grad^2 \Lyap(x)) + b_\circ(x,\zeta)\cdot\grad \Lyap(x)
\,\le\, C_0'(\epsilon) - (C_1-\epsilon)\theta \Lyap(x)\,.
\end{equation}
On the other hand, as shown in \cite[Lemma~5.1]{APS19},
if $\theta<\theta_\circ$, then
$x\mapsto\abs{x}^{1-\theta}\fJ[\Lyap](x)$
vanishes at infinity if $\theta\in[1,2)$ and
$\fJ[\Lyap](x) \sim \abs{x}^{\theta-2}$ if $\theta\ge2$.
Thus, \cref{Ex1.2B} shows that \cref{EA1.1A} is satisfied for any
$\gamma<C_1 \theta_\circ$
We mention parenthetically that \cref{EA1.1A} holds for a Lyapunov function
taking the form above in the case of multiclass queueing networks in
the Halfin--Whitt regime with reneging \cite[Theorem~3.5]{APS19}.

Examining the proof of \cite[Lemma~5.1]{APS19}, the estimates of the growth of
$\fJ[\Lyap](x)$ depend only on the values
of $\int_{\sB_r^c}\abs{z}^{\theta}\nu(\D{z})$ and
$\int_{\sB\setminus\{0\}}\abs{z}^{2}\,\nu(\D{z})$.
Therefore, scaled versions of these estimates can be derived to address general measures
$\nu(x,\D{z})$ encountered here.
But it should be clear from the preceding discussion that, in general, if
$\sup_{\zeta\in\Act} b_\circ(x,\zeta)$ has strict
sublinear growth in $x$, and $\nu$ is heavy-tailed, that is,
$\int \abs{z}^{\theta_\circ}\nu(\D{z})=\infty$ for some $\theta_0>0$,
then \cref{EA1.1B} cannot be satisfied.
This is because, according to \cite[Corollary~5.3]{APS19}, the resulting process
is at most strictly subgeometrically ergodic, whereas \cref{EA1.1B}
implies geometric ergodicity.
\end{example}

We are now ready to state one of our main results,
whose proof is in \cref{S-risk}.
This is \cref{T1.1} below, which establishes the existence of
an optimal stationary Markov control, as well as verification of optimality.
Recall the definition of $\order(f)$ from \cref{Snot}.

\begin{theorem}\label{T1.1}
Grant \hyperlink{A1}{\ttup{A1}}--\hyperlink{A3}{\ttup{A3}} and \cref{A1.1}.
Then the following hold.
\begin{itemize}
\item[(a)]
There exists a unique positive function
$V\in\Sobl^{2,p}(\Rd)\cap\order(\Lyap)$ satisfying
\begin{equation*}
\cI V(x) \,=\, \sE^* \, V(x)\ \ \text{a.e.\  in\ }\Rd\,,
\quad V(0)=1\,.
\end{equation*} 
\item[(b)]
There exists a measurable $v\colon\Rd\to\Act$ satisfying
\begin{equation}\label{ET1.1B}
 I_v[V, x] + b_v(x)\cdot\grad V(x) + c_v(x) V(x)
\,=\, \min_{\zeta\in\Act}\,\bigl\{I[V, x, \zeta]
+ b(x,\zeta)\cdot \grad V(x) + c(x,\zeta) V(x)\bigl\}\,.
\end{equation}
\item[(c)]
A stationary Markov control $v$ is optimal
in the sense of \cref{D1.1}
if and only if it satisfies \cref{ET1.1B} a.e.\ in $\Rd$.
Moreover, for such a control $v$ we have
 $\sE_x(c,v)=\sE^*$ for all $x\in\Rd$, or in other words,
the risk-sensitive value does not depend on the initial condition $x\in\Rd$.
\end{itemize}
\end{theorem}

\begin{remark}
We say that $\nu$ has locally compact support if for every $r>0$ there
exist $R=R(r)$ such that
$\nu(x,\zeta,\sB_R^c)=0$ for all $x\in\sB_r$ and $\zeta\in \Act$.
Concerning \cref{T1.1}, if
$x\to c(x,\zeta)$ is locally H\"older continuous, and $\nu$ has
locally compact support,
then $V\in C^{2,\delta}(\Rd)\cap\order(\Lyap)$, for some $\delta\in(0,1)$,
by elliptic regularity.
To see this, let $f$ be a Lipschitz-continuous function on $\Rd$, with
Lipschitz constant $\mathrm{Lip}_f$.
Then, using \cref{EA2,E-barnu} and the Cauchy--Schwarz inequality, we obtain
\begin{equation*}
\begin{aligned}
&\babss{\int_\Rd f(x+z)\nu(x,\zeta,\D{z}) - \int_\Rd f(y+z)\nu(y,\zeta,\D{z})}\\
&\mspace{50mu}\,\le\,
\int_{\RR^m\setminus\{0\}}
\babs{f\bigl(x+g(x,\zeta,\xi)\bigr)-f\bigl(y+g(y,\zeta,\xi)\bigr)}\,\Uppi(\D\xi)\\
&\mspace{50mu}\,\le\,\Bar\nu\,\mathrm{Lip}_f\,\bigl(1+\sqrt{C_R}\bigr)\,\abs{x-y}
\qquad\forall\,x,y\in\sB_r\,.
\end{aligned}
\end{equation*}
This shows that the map
$f\mapsto I[f,\cdot\,, \zeta]$ preserves local Lipschitz continuity, uniformly in $\zeta$.
Then the regularity of $V$ follows from standard elliptic theory.
In the case when $\nu$ does not have locally compact support, sufficient conditions
for the regularity of $V$ can be obtained by combining Lemma~5.3 and Theorem~5.3
of \cite{AB-20a}.
\end{remark} 

\begin{remark}
One can also consider a risk-sensitive maximization problem under the
assumptions of \cref{T1.1}. This can be done using the results of
\cref{S-risk} together with the approach of \cite[Theorem~3.1]{BS18}.
\end{remark}

\begin{remark}
The results of \cref{T1.1} also hold under more general hypotheses. For instance,
if we restrict our class of controls to the set of all
stationary Markov controls and assume
that $x\mapsto \int_{A} \frac{\abs{z}}{1+\abs{z}^2} \nu(x, \D{z})$ is continuous and
bounded for all $A\in\fB(\RR^m\setminus\{0\})$, then it is known that the martingale
problem is well-posed
and the family of martingale solutions are strong Markov \cite[Theorem~4.3]{ST75}.
All the results of this article hold in this set up. 
\end{remark}

\begin{remark}
The risk-sensitive minimization problem with a near-monotone hypothesis
on the running cost
is also of interest.  Here, we can replace the blanket stability in
\cref{A1.1} with a \emph{stabilizability} hypothesis, namely that
\cref{EA1.1B} holds under some Markov control.
Then existence of a principal eigenfunction $V$ on the whole space can be shown.
However, asserting that the eigenvalue equals $\sE^*$ and the verification of
optimality results require additional hypotheses; see \cite{AB18}.
\end{remark}

We also consider a risk-sensitive maximization problem without the
blanket stability hypotheses in \cref{A1.1}.
This assumption is replaced by a
near-monotone hypothesis on the running cost
(see \hyperlink{H}{(H)} in \cref{S-max}), which
penalizes the growth of the
process at infinity under any optimal control.
Our main result on the maximization problem, which also
requires \hyperlink{A4}{(A4)} and \hyperlink{H}{(H)} in \cref{S-max},
can be roughly stated as follows.

\begin{theorem}[Informal statement]
Under \hyperlink{A1}{\ttup{A1}}--\hyperlink{A4}{\ttup{A4}} and
\hyperlink{H}{\ttup{H}},
results analogous of \cref{T1.1} hold for the maximization problem
in \cref{E-max2}.
\end{theorem}

For detailed statements and their proofs we refer to \cref{S-max}.

In concluding this section, we discuss the difficulties encountered in
extending the results to the case where $\Pi(\RR^m\setminus\{0\})=\infty$.
To our understanding, this class of problems are more
difficult because of the hurdles appearing in the
study of associated eigenvalue problems and nonlinear Dirichlet problems.
Note that when $\Pi$ has finite mass we can treat the nonlocal term as a \emph{zeroth} order
perturbation, but the same is not true when $\Pi$ is not finite.

There is a large body of work dealing with nonlinear nonlocal operators where the nonlocal
kernel resembles the fractional Laplacian. These operators are called stable-like
by Bass \cite{Bass-09}.
However, whereas Harnack estimates have been established for such
operators (see \cite{Caff-Silv-09}), they do not cover the case
when a zeroth order term, that is, the term due to $c$, is present.

In summary then, we can distinguish two main classes of operators.
The first, is the class of operators studied in this paper, which contain
a non-degenerate Brownian motion term and a nonlocal term with finite characteristic
measure, but whose kernel has no particular regularity and could be singular with
respect to the Lebesgue measure.
As mentioned earlier, for stochastic networks in the Halfin--Whitt regime, these operators
correspond to networks with asymptotically negligible service interruptions.

For the second class, we replace
$\trace(a\grad^2 f(x)) + \fJ[f](x)$ in \cref{Ex1.2Z} with a stable-like operator.
For such operators, the nonlocal term has a nice density, and the resulting process is
open set irreducible \cite[Theorem~3.1]{APS19}.
Such operators are encountered in stochastic networks with heavy-tailed arrivals
\cite[Section~4.1]{APS19}.
Risk-sensitive control for this class of systems remains an open problem
primarily due to the lack of Harnack estimates mentioned earlier.

\section{The eigenvalue problem in bounded domains}\label{S-bounded}

In this section we consider the principal eigenvalue problem for nonlocal
operators on bounded domains and establish several properties.
The assumptions here are more general than
\hyperlink{A1}{(A1)}--\hyperlink{A3}{(A3)}, and the proofs
are purely analytical, and devoid of probabilistic arguments.
These results are crucial for the
study of the risk-sensitive control problems appearing later in the paper. 
The proofs of the results stated in this section can be
found in \cref{S-proofs}.

Let $D$ be a bounded smooth domain in $\Rd$.
Without any loss of generality we may assume that $0\in D$.
Let us point out that  compactness of $\Act$ and nonnegativity of
$c$ are not required in this section.
We define the (uncontrolled) linear operators
 $\sA$ and $\sA^c$ by
\begin{equation}\label{EsA}
\begin{aligned}
\sA f(x) &\,\df\,
\trace\bigl(a(x)\grad^2 f(x)\bigr) + I[f,x] +b(x)\cdot \grad u(x) \,,\\
\sA^c f(x) &\,\df\, \sA f(x) + c(x) f(x)\,,
\end{aligned}
\end{equation}
with $I[f,x]$ given by \cref{E-I}, with $\nu$ not depending on the parameter $\zeta$.

\Cref{A2.1} which follows, is enforced throughout this section,
without further mention.
\begin{assumption}\label{A2.1}
The following hold.
\begin{enumerate}
\item[(1)]
The map $x\mapsto a(x)$ is continuous in $\Bar D$, and
there exists a positive constant $\upkappa$ such that
$\upkappa \mathbb{I} \le a(x) \le \upkappa^{-1} \mathbb{I}$
for all $x\in\Bar D$,
where $\mathbb{I}\in\RR^{d\times d}$ denotes the identity matrix.

\item[(2a)] For the operators $\sA$ ad $\sA^c$ in \cref{EsA}:
$b\colon D\to\Rd$ and $c\colon D\to\RR$ are Borel measurable and bounded,
and $x\mapsto\nu(x,\Rd)$ is locally bounded.

\item[(2b)] For the operator $\cI$ in \cref{E-cI}:
$b\colon D\times\Act\to\Rd$ and $c\colon D\times\Act\to\RR$ are continuous and bounded,
and $(x, \zeta)\to \sup_{\zeta\in\Act}\nu(x, \zeta, \Rd)$ is locally bounded.
\end{enumerate}
\end{assumption}

Next, we define the generalized Dirichlet principal eigenvalue $\lambda_D$
of $\cI$ (or $\sA^c$) on a domain $D$.
Let  $\cC_{b,+}(\Rd)$
denote the cone in $\cC_{b}(\Rd)$ 
consisting of nonnegative functions.
We define
\begin{equation*}
\Uppsi^+(\lambda) \,\df\, \bigl\{\psi\in \cC_{b,+}(\Rd)\cap\Sobl^{2, d}(D)\,\colon\,
\psi>0 \text{\ in\ } D\,,\ \cI\psi(x) -\lambda\psi\le 0 \text{\ in\ } D\bigr\}\,,
\end{equation*}
and let
\begin{equation}\label{E-lamD}
\lambda_D(\cI) \,\df\, \inf\,\bigl\{\lambda\in\RR\,\colon\,
\Uppsi^+(\lambda)\ne \varnothing\bigr\}\,.
\end{equation}
The eigenvalue $\lambda_D(\sA^c)$ is defined in the same manner.
 
The first main result of this section is the following.
Its proof relies on the
nonlinear Krein--Rutman theorem in \cite{A18}.

\begin{theorem}\label{T2.1}
Let $D$ be a $\cC^{1,1}$ bounded domain in $\Rd$.
There exists a unique $\psi_D\in \cC_{b,+}(\Rd)\cap\Sobl^{2,p}(D)$, $p>d$,
satisfying
\begin{align*}
\cI \psi_D &\,=\, \lambda_D(\cI)\, \psi_D \quad \mbox{in\ } D\,,\\
\psi_D &\,=\, 0\quad \text{in\ } D^c\,,\\
\psi_D &\,>\, 0 \quad \text{in\ } D\,, \quad \psi_D(0)=1\,.
\end{align*}
Moreover, if $u\in \cC_{b,+}(\Rd)\cap\Sobl^{2,p}(D)$, $p>d$,
is positive in $D$ and satisfies
\begin{equation*}
\cI u \,\le\, \lambda \, u \quad \text{in\ } D\,,
\end{equation*}
for some $\lambda\in \RR$,
then, either $\lambda>\lambda_D(\cI)$, or $\lambda=\lambda_D(\cI)$ and
$u=\kappa \psi_D$ for some $\kappa>0$.
In addition, the  assertions above hold for the operator $\sA^c$.
\end{theorem}

We refer to $\psi_D$ as the \emph{principal eigenfunction} of $\cI$ on $D$,
and to $(\lambda_D,\psi_D)$ as the \emph{principal eigenpair}.
If the operator is not specified,
$\lambda_D$ refers to the principal eigenvalue
of $\cI$ or $\sA^c$.

As a corollary to the proof of \cref{T2.1} we obtain the following monotonicity property
with respect to the domain.

\begin{corollary}\label{C2.1}
Suppose that $D\subsetneq D'$.
Then we have $\lambda_{D}<\lambda_{D'}$.
\end{corollary}

We next address the continuity properties with respect to the domain $D$.
Let $\{D_n\}_{n\in\NN}$ be a decreasing
sequence of smooth domains whose intersection is $D$, and which
satisfies an exterior sphere condition uniformly in $n\in\NN$,
that is, there exists $r>0$ such that for all large $n$,
every point of $\partial D_n$ can be touched from outside of $D_n$ with a ball
of radius $r$.

\begin{theorem}\label{T2.2}
Let $D_n\to D$ as above. Then $\lambda_{D_n}\to \lambda_D$ as $n\to\infty$.
\end{theorem}

In the following theorem,
we incorporate the dependence of $\lambda_D$ on $c$ explicitly in the notation,
by writing this as $\lambda_D(c)$.

\begin{theorem}\label{T2.3}
For any two potentials $c$ and $c'$ the following hold.
\begin{enumerate}
\item[\ttup{i}]
If $c\le c'$, and $c'>c$ on a subset of $D$ with positive Lebesgue measure,
then $\lambda_D(c)<\lambda_D(c')$.

\item[\ttup{ii}]
For the operator $\sA^c$, we have 
$\lambda_D\bigl(\theta c + (1-\theta) c'\bigr)
\,\le\,\theta \lambda_D(c) + (1-\theta) \lambda_D(c')$
for all $\theta\in[0,1]$.
\end{enumerate}
\end{theorem}

\section{A controlled eigenvalue problem}\label{S-control}

In this section we characterize the maximal exit rate probability 
\begin{equation}\label{E-exit}
\Theta_D\,\df\, \sup_{Z\in\Uadm}\, \limsup_{T\to\infty}\,\frac{1}{T}\,
\log \Prob^Z_x(\uptau(D)>T)
\end{equation}
for the jump diffusion model in \cref{E-sde}
on a bounded $\cC^{1,1}$ domain $D$.

This topic has a long history in the context of continuous diffusions
in the uncontrolled
\cite{DV76,Gong-88,Pinsky-85} and controlled \cite{Fleming-85,BB10,BA15} settings,
and is linked to the general theory of quasi-stationary distributions
\cite{Champagnat-16}.

We assume \hyperlink{A1}{(A1)}--\hyperlink{A3}{(A3)}.
Let $c=0$, and denote
the corresponding operator $\cI$ in \cref{E-cI} as $\widetilde\cI$, that is,
\begin{equation*}
\widetilde\cI f(x) \,=\,
\trace\bigl(a(x)\grad^2 f(x)\bigr)  
+ \inf_{\zeta\in\Act}\bigl\{ I[f,x, \zeta] + b(x,\zeta)\cdot\grad f(x)\bigr\}\,.
\end{equation*}
For a given $u\in\cC_0(D)$ we can define $v=\widetilde\cT u$ to be the solution of
\begin{equation*}
\widetilde\cI (-v) \,=\, u \quad \text{in\ } D\,, \quad \text{and}\quad v=0
\text{\ \ in\ } D^c\,.
\end{equation*}
Then we can apply the tools of \cref{S-bounded} on $\widetilde\cT$
(see also, \cref{S-proofs}) to obtain the following.

\begin{theorem}\label{T3.1}
The following hold.
\begin{itemize}
\item[(a)]
There exists a unique $\Tilde\psi_D\in \cC_{b}(\Rd)\cap\Sobl^{2,p}(D)$, $p>d$,
and $\Tilde\lambda_D<0$ satisfying
\begin{equation}\label{ET3.1A}
\begin{aligned}
\widetilde\cI \Tilde\psi_D &\,=\, \Tilde\lambda_D\, \Tilde\psi_D \text{\ \ in\ } D\,,\\
\Tilde\psi_D &\,=\, 0\quad \text{in\ } D^c\,,\\
\Tilde\psi_D &\,<\, 0 \quad \text{in\ } D\,, \quad \Tilde\psi_D(0)=-1\,.
\end{aligned}
\end{equation}

\item[(b)]
If $D_n\to D$ in the sense of \cref{T2.2}, then
$\lim_{n\to\infty} \Tilde\lambda_{D_n}=\Tilde\lambda_D$.
\end{itemize}
\end{theorem}

The main result in this section is \cref{T3.2} which asserts that
$\Theta_D=\Tilde\lambda_D$.
As before, $\Uadm$ denotes the set of admissible controls and the dynamics are
given by \cref{E-sde}.
We need the following version of It\^{o}'s formula which plays
a crucial role in this study.

\begin{lemma}\label{L3.1}
Let $\uptau$ be the first exit time from a bounded domain $D$, and $c$ be bounded in $D$.
Then for any $u\in\Sob^{2,p}(D)\cap\cC_b(\Rd)$, $p>d$, we have
\begin{equation}\label{EL3.1A}
\Exp_x\Bigl[\E^{\int_0^{\uptau\wedge t} c(X_s,Z_s)\, \D{s}}\,
u(X_{\uptau\wedge t})\bigr]-u(x)
\,=\, \Exp_x
\left[\int_0^{\uptau\wedge t} \E^{\int_0^s c(X_r,Z_r)\, \D{r}}\,
\cA^c u(X_s,Z_s)\, \D{s}\right],
\quad t\ge 0\,,
\end{equation}
where $\cA^c$ is as in \cref{EcA}.
\end{lemma}

\begin{proof}
We follow the technique of Krylov \cite{Krylov}.
Consider a sequence of bounded, smooth functions
$u_m$ such that
\begin{equation*}
\norm{u_m-u}_{\Sob^{2,p}(D)}\to 0\,,
\quad \norm{u_m-u}_{\cC_b(\Rd)}\to 0\,,
\quad \text{as\ } m\to \infty\,.
\end{equation*}
By It\^{o}'s formula we then have 
\begin{equation}\label{EL3.1B}
\Exp_x \Bigl[\E^{\int_0^{\uptau\wedge t} c(X_s,Z_s)\, \D{s}}\,
u_m(X_{\uptau\wedge t})\Bigr]
-u_m(x)
\,=\, \Exp_x\left[\int_0^{\uptau\wedge t} \E^{\int_0^s c(X_r,Z_r)\,\D{r}}\,
\cA^c u_m(X_s,Z_s)\,\D{s}\right], \quad t\ge 0\,.
\end{equation}
By the compactness of
the embedding $\Sob^{2,p}(D)\hookrightarrow \cC^{1,\alpha}(D)$, it is easily seen that
 as $m\to\infty$, the following holds.
\begin{multline*}
\Exp_x\biggl[\int_0^{\uptau\wedge t} \Bigl(\babs{I[u_m, X_s, Z_s]-I[u, X_s, Z_s]}
+ \babs{u_m(X_s)-u(X_s)}\\ + \babs{\grad u_m(X_s)-\grad u(X_s)}\Bigr)\,\D{s}
\biggr] \,\xrightarrow[m\to\infty]{}\, 0\,.
\end{multline*}
Thus in order to pass to the limit in \cref{EL3.1B} to obtain \cref{EL3.1A},
we only need to verify the passage to the limit for the term
\begin{equation*}
\Exp_x\left[\int_0^{\uptau\wedge t} \trace (a\grad^2 u_m)(X_s)\,\D{s}\right].
\end{equation*}
To verify this limit it is enough to show that
\begin{equation}\label{EL3.1C}
\left|\Exp_x\left[\int_0^{\uptau\wedge t} f(X_s)\,\D{s}\right]\right|
\,\le\, \kappa \norm{f}_{L^p(D)}
\end{equation}
for some constant $\kappa$ not depending on $f$.
It is also enough if we prove this for functions $f$
that are nonnegative and Lipschitz in $D$.
Let $w\in\cC^{2,\alpha}(D)$ be the unique solution to
\begin{equation}\label{EL3.1D}
\trace (a\grad^2 w) \,=\, f\quad \text{in\ } D\,, \quad w=0\text{\ \ on\ } \partial D\,.
\end{equation}
Applying the maximum principle \cite[Theorem 9.1]{GilTru} we have
$\sup_{D}\abs{w}\le \kappa \norm{f}_{L^d(D)}$.
By the Sobolev estimate \cite[Theorem~9.14]{GilTru} we then have
\begin{equation}\label{EL3.1E}
\norm{w}_{\Sob^{2,p}(D)} \,\le\, \kappa_1 \norm{f}_{L^p(D)}
\end{equation}
for some constant $\kappa_1$.
Thus, by Sobolev embedding and \cref{EL3.1E}, we obtain
\begin{equation}\label{EL3.1F}
\norm{w}_{\cC^{1,\alpha}(D)} + \norm{w}_{\cC_b(\Rd)} \,\le\,
\kappa_2 \norm{w}_{\Sob^{2,p}(D)} \,\le\, \kappa_2 \kappa_1 \norm{f}_{L^p(D)}
\end{equation}
for some $\alpha\in(0,1)$.
Therefore, applying It\^{o}'s formula to \cref{EL3.1D}, we deduce that
\begin{align*}
\Exp_x\left[\int_0^{\uptau\wedge t} f(X_s)\,\D{s}\right] &\,=\,
\Exp_x\left[\int_0^{\uptau\wedge t} \trace (a\grad^2 w) (X_s)\,\D{s}\right]\\
&\,=\, \Exp_x \bigl[w(X_{\uptau\wedge t})\bigr]- w(x)
- \Exp_x\left[\int_0^{\uptau\wedge t} I[w,X_s, Z_s] + b(X_s,Z_s)\cdot \grad w(X_s)\,\D{s}
\right]\\
&\,\le\, \kappa_3 \norm{f}_{L^p(D)}
\end{align*}
 for some constant $\kappa_3$ by \cref{EL3.1F}.
This establishes \cref{EL3.1C}, and completes the proof.
\end{proof}

\begin{remark}
More generally, \cref{EL3.1A} is valid for any
$u\in\Sobl^{2,d}(D)$ satisfying $I[u,\cdot\,]\in\Sobl^{2,d}(D)$,
provided $g$ is independent of $\zeta$.
This can be shown by following the argument in
\cite[Lemma~4.1]{ACPZ18}.
\end{remark}

Recall the definition in \cref{E-exit}.
The main result of this section is the following.

\begin{theorem}\label{T3.2}
It holds that $\Theta_D=\Tilde\lambda_D$.
In addition, a stationary Markov control is optimal for the problem in
\cref{E-exit} if and only if it is an a.e.\ measurable selector of \cref{ET3.1A}.
\end{theorem}

\begin{proof}
We first show sufficiency.
Let $v$ be a measurable selector of \cref{ET3.1A}, that is,
\begin{equation*}
\cA_v\Tilde\psi_D(x)\,\df\,\trace(a \grad^2 \Tilde\psi_D)(x)
+ I_v[\Tilde\psi_D, x]+  b_v(x)\cdot\grad \Tilde\psi_D(x)
\,=\, \Tilde\lambda_D \Tilde\psi_D\quad \text{a.e.\ in\ } D\,.
\end{equation*}
With $\psi_D=-\Tilde\psi_D$, we get
\begin{equation}\label{ET3.2A}
\cA_v\psi_D(x) 
\,=\, \Tilde\lambda_D \psi_D(x)\quad \text{a.e.\ in\ } D\,.
\end{equation}
Now consider a collection of smooth, increasing domains $D_k\Subset D$
such that $\cup_k D_k=D$. Let $\uptau_k$ be the first exit time from $D_k$.
Note that $\uptau_k\le \uptau$ for all $k$.
Since $\psi_D\in\Sob^{2,p}(D_k)$ we can apply \cref{L3.1} to \cref{ET3.2A}
to obtain
\begin{align*}
\psi_D(x) &\,=\, \Exp^v_x\left[\E^{-\Tilde\lambda_D t\wedge\uptau_k}
\psi_D(X_{t\wedge\uptau_k})\right]
\\
&\,\le\, \E^{-\Tilde\lambda_D t}\norm{\psi_D}_\infty \Prob_x^v(t<\uptau_k)
+ \biggl(\sup_{x\in D^c_k}\,\psi_D\biggr)\,\E^{-\Tilde\lambda_D t}
\Prob^v_x(\uptau_k\le t) \\
&\,\le\, \E^{-\Tilde\lambda_D t}\norm{\psi_D}_\infty \Prob_x^v(t<\uptau)
+ \biggl(\sup_{x\in D^c_k}\,\psi_D\biggl)\,
\E^{-\Tilde\lambda_D t}\Prob^v_x(\uptau_k\le t)
\\
&\,\xrightarrow[k\to\infty]{}\,
\E^{-\Tilde\lambda_D t}\norm{\psi_D}_\infty \Prob_x^v(t<\uptau)\,.
\end{align*}
We take logarithms on both sides, divide by $t$, and let $t\to\infty$ to obtain
\begin{equation}\label{ET3.2B}
\Tilde\lambda_D \,\le\, \limsup_{t\to\infty}\, \frac{1}{t}\,\log \Prob^v_x(\uptau>t)
\,\le\, \Theta_D\,.
\end{equation}
Next, consider a domain $D_n\Supset D$.
Let $(\Tilde\psi_n, \Tilde\lambda_n)=(\Tilde\psi_{D_n}, \Tilde\lambda_{D_n})$
denote the corresponding eigenpair given in \cref{T3.1}.
Then, with $\psi_n\df-\Tilde\psi_n$, we have
\begin{equation*}
\trace(a \grad^2 \psi_n)(x) + \sup_{\zeta\in\Act}\,\left\{ I[\psi_n, x, \zeta]
+  b(x,\zeta)\cdot\grad \psi_n(x)\right\} \,=\,
 \Tilde\lambda_n \psi_n(x)\,.
 \end{equation*}
Applying \cref{L3.1} in the domain $D$, and using the fact $\psi_n>0$ in $D_n$,
we see that for any $Z\in\Uadm$ we have
\begin{align*}
\psi_n(x) &\,\ge\, \Exp^Z_x\left[\E^{-\Tilde\lambda_n t\wedge\uptau}\,
\psi_n(X_{t\wedge\uptau})\right]\\
& \,\ge\, \E^{-\Tilde\lambda_n t}\, \biggl(\min_{D}\,\psi_n\biggr)
\, \Prob^Z_x(\uptau>t)\,,
\end{align*}
from which we obtain
\begin{equation*}
\Tilde\lambda_n \,\ge\, \limsup_{t\to\infty}\, \frac{1}{t}\,\log \Prob^Z_x(\uptau>t)\,.
\end{equation*}
Since $Z$ is arbitrary, we have $\Tilde\lambda_n\ge \Theta_D$,
and thus letting $n\to\infty$
and applying \cref{T3.1}\,(b), we obtain $\Tilde\lambda_D\ge \Theta_D$.
Combining this with \cref{ET3.2B} we have
$\Tilde\lambda_D=\Theta_D$. Since $\Tilde\lambda_D$ is the
principal eigenvalue corresponding to any minimizing selectors,
we have thus shown sufficiency.

We next prove necessity.
Let $v$ be a optimal stationary Markov control, that is, it satisfies
\begin{equation*}
\Theta_D \,=\, \limsup_{T\to\infty}\, \frac{1}{T}\,\log \Prob^v_x(\uptau>T)\,.
\end{equation*}
Let $(\theta_v, u)$ be a solution of 
\begin{equation*}
\cA_v u
\,=\, \theta_v u \quad \text{a.e.\ in\ } D\,,
\end{equation*}
with $u>0$ in $D$, and $u=0$ on $D^c$.
Using the above arguments  we obtain 
\begin{equation*}
\theta_v \,=\, \limsup_{T\to\infty}\, \frac{1}{T}\,
\log \Prob^v_x(\uptau>T) \,=\, \Theta_D\,.
\end{equation*}
By \cref{ET3.1A} we have
\begin{equation*}
\cA_v\psi_D(x)
\,\le\, \Tilde\lambda_D \psi_D(x)
\quad \text{a.e.\ in\ } D\,,
\end{equation*}
and $\psi_D=-\Tilde\psi_D>0$ in $D$.
Since $\theta_v=\Tilde\lambda_D$,  it follows from \cref{T2.1} that
$\psi_D=\kappa u$ for some $\kappa>0$.
Therefore, $v$ is a minimizing selector.
This completes the proof.
\end{proof}

\section{Risk-sensitive control}\label{S-risk}

In this section, we study the risk-sensitive control problem in
\cref{D1.1} for the controlled diffusion in
\cref{E-sde}, and characterize optimality via the
risk-sensitive HJB equation in \cref{E-HJB}.
Hypotheses \hyperlink{A1}{(A1)}--\hyperlink{A3}{(A3)}
are in full effect in this section, without further mention.
In \cref{S4.1} we study the eigenvalue problem for a linear operator,
and use these results in \cref{S4.2} which is devoted to the proof of
\cref{T1.1}.

\subsection{The eigenvalue problem in \texorpdfstring{$\Rd$}{}}\label{S4.1}

Recall that $\Usm$ denotes the set of stationary Markov controls.
As in the proof of \cref{T3.2}, by $\cA_v$, with $v\in\Usm$, we denote the linear operator
\begin{equation*}
\cA_v f (x)\,\df\, \trace (a\grad^2 f)(x) + I_v[f,x] + b_v(x)\cdot \grad f(x),
\quad f\in \cC^2(\Rd)\cap \cC_b(\Rd)\,,
\end{equation*}
where we use the notation in \cref{Esimpl}.
In analogy to the notation in \cref{EsA}, we define
$\cA_v^c \df \cA_v + c_v$.
We also use the notation
\begin{equation}\label{Enot4.1}
\sE_x^v(c) \,\df\, \limsup_{T\to\infty}\, \frac{1}{T}\,
\log \Exp_x^v\left[\E^{\int_0^T c_v(X_s)\,\D{s}}\right],
\quad\text{and\ } \sE^v(c)\,\df\, \inf_{x\in\Rd}\,\sE_x^v(c)\,,\quad v\in\Usm\,.
\end{equation}
In the first part of this section we characterize $\sE^v(c)$ as a principal
eigenvalue of the  operator $\lambda_\Rd(\cA_v^c)$ in $\Rd$ (see \cref{E-lamD}).

We keep in mind that, by \cref{A1.1},
the operator $\cA_v$ satisfies
$\cA_v\Lyap \le \widehat{C} \Ind_{\cK}-\gamma \Lyap$ if $c$ is bounded,
and $\cA_v\Lyap \le \widehat{C} \Ind_{\cK}-\ell \Lyap$ otherwise.
Also, for any fixed $v\in\Usm$, the operator $\cA_v^c$ satisfies
the hypotheses for $\sA^c$ in
\cref{A2.1}, and thus the results of \cref{S-bounded} apply.
The main theorem in this section is the following.

\begin{theorem}\label{T4.1}
Grant \cref{A1.1}.
The following hold for each $v\in\Usm$.
\begin{itemize}
\item[(a)]
There exists a positive $\Psi_v\in \Sobl^{2,p}(\Rd)\cap\order(\Lyap)$, $p>d$,
and $\lambda^v\in\RR$ satisfying
\begin{equation}\label{ET4.1A}
\cA_v\Psi_v + c_v\, \Psi_v \,=\, \lambda^v \Psi_v\quad \text{in\ } \Rd\,,
\quad\text{and\ } \Psi_v(0)=1\,.
\end{equation}
\item[(b)]
$\lambda^v=\lambda_\Rd(\cA_v^c)=\sE^v_x(c)$ for all $x\in\Rd$.
\item[(c)]
The function $\Psi_v$ is the unique solution of \cref{ET4.1A} 
in $\Sobl^{2,d}(\Rd)\cap\order(\Lyap)$ with $\lambda^v=\lambda_\Rd(\cA_v^c)$.
\end{itemize}
\end{theorem}

Note that $I_v[\Psi_v, x]$ is well defined since $\Psi_v\in\order(\Lyap)$
by \cref{A1.1}.

The rest of this section is devoted to the proof of \cref{T4.1}.
The key steps involved in the proof are as follows:
(i) we start with the Dirichlet eigenvalue problems on a sequence of balls
increasing to $\Rd$ and then justify the passage of limit in the equations;
(ii) we  show that the limits of the principal eigenvalues on balls coincide
with $\sE^v_x(c)$.
We break down the proof in several lemmas. 
One of the key lemmas is \cref{L4.1} where we obtain
a stochastic representation for the Dirichlet principal eigenfunctions
of $\cA^c_v$.
This representation (see \cref{EL4.1A}) is the main ingredient in obtaining uniform
local  bounds for the eigenfunctions when we consider a sequence of domains
increasing to $\Rd$. This we do in \cref{L4.3} where a stochastic representation
is obtained for the principal eigenfunction of $\cA^c_v$ in $\Rd$.
\cref{L4.2} establishes a lower bound on the limits of the principal eigenvalues
of $\cI$ over a sequence of increasing domains,
which is required
to prove the stochastic representation in \cref{L4.3}.
In \cref{L4.5} we show that the principal eigenvalue of $\cA^c_v$ on $\Rd$ coincides
with the risk sensitive value $\sE^v_x$ with respect to $c_v$.

In the interest of economy of notation,
we adopt the following definitions.

\begin{definition}
For $v\in\Usm$, we let
$$\lambda^v_n\,\df\, \lambda_{\sB_n}\bigl(\cA^c_v\bigr)\,,\qquad
\text{and\ \ } \lambda^v\,\df\, \lim_{n\to\infty}\, \lambda^v_n\,.$$
Also,
$$\lambda_n\,\df\, \lambda_{\sB_n}(\cI)\,,\qquad
\text{and\ \ } \lamstr\,\df\, \lim_{n\to\infty}\, \lambda_n\,.$$
Here, as defined in \cref{Snot},
$\sB_n$ denotes the open ball of radius $b$ centered at $0$,
and $\lambda_{\sB_n}$ is as in \cref{E-lamD}.
\end{definition}

We begin with the following stochastic representation formula, inspired from
\cite{AB18, ABS19}, for the Dirichlet eigenfunction in the bounded domain.
Recall from \cref{T2.1} that the Dirichlet eigenfunctions
in $D$ belong to $\cC_0(D)\cap\Sob^{2,p}(D)$.
This can also be seen by \cite[Theorem 9.15]{GilTru}.

\begin{lemma}\label{L4.1}
For $v\in\Usm$, 
let $(\psi^v_n,\lambda^v_n)\in\cC_0(\sB_n)\cap\Sob^{2,p}(\sB_n)\times\RR$, $p>d$,
be the Dirichlet principal eigenpair satisfying
\begin{align*}
\cA_v\psi^v_n + c_v \psi^v_n&\,=\, \lambda^v_n\, \psi^v_n \quad \text{in\ } \sB_n\,,\\
\psi^v_n &\,=\,0\quad \text{in\ } \sB_n^c\,,\\
\psi^v_n &\,>\, 0 \quad \text{in\ } \sB_n\,.
\end{align*}
Then for every $r\in (0, n)$ we have
\begin{equation}\label{EL4.1A}
\psi^v_n(x) \,=\, \Exp^v_x\left[ \E^{\int_0^{\uuptau_r} (c_v(X_s)-\lambda^v_n)\,\D{s}}\,
\psi^v_n(X_{\uuptau_r})\Ind_{\{\uuptau_r<\uptau_n\}}\right]
\quad \forall\,x\in \sB_n\setminus \overline{\sB}_r\,,
\end{equation}
where $\uuptau_r=\uptau(\sB^c_r)$ as defined in \cref{Snot}.
In addition, for all $n\in\NN$, we have $\lambda^v_n\le \sE^v_x(c)$ for $x\in\sB_n$.
\end{lemma}

\begin{proof}
Applying \cref{L3.1} and using the fact $\psi^v_n=0$ in $\sB_n^c$, it follows that
\begin{equation}\label{EL4.1B}
\psi^v_n(x) \,=\, \Exp_x\left[ \E^{\int_0^{t\wedge \uuptau_r} (c_v(X_s)-\lambda^v_n)\,\D{s}}\,
\psi^v_n(X_{t\wedge \uuptau_r})\Ind_{\{\uuptau_r\wedge t <\uptau_n\}}\right],
\quad t\ge 0\,.
\end{equation}
Letting $t\to\infty$ in \cref{EL4.1B},
and applying Fatou's lemma, we obtain
\begin{equation}\label{EL4.1C}
\psi^v_n(x) \,\ge\, \Exp_x\left[ \E^{\int_0^{\uuptau_r} (c_v(X_s)-\lambda^v_n)\,\D{s}}\,
\psi^v_n(X_{\uuptau_r})\Ind_{\{\uuptau_r <\uptau_n\}}\right].
\end{equation}
Define
\begin{equation*}
\Hat{c}_v(x) \,\df\, - \Ind_{\sB_r}(x) + c_v(x) \,,
\end{equation*}
and let $(\Hat\psi^v_n, \Hat\lambda^v_n)$ be the principal Dirichlet
eigenpair of the operator $\cA_v^{\Hat{c}}\df \cA_v+ \Hat{c}_v$ in $\sB_n$.
Then from \cref{T2.3} it follows that
$\lambda^v_n>\Hat\lambda^v_n$ and therefore, by \cref{T2.2}, we can find a ball
$\sB_R$ with $R>n$ such that the principal Dirichlet eigenpair
$(\Hat\psi^v_R, \Hat\lambda^v_R)$ of $\cA_v^{\Hat{c}}$ in $\sB_R$ satisfies 
$\lambda^v_n>\Hat\lambda^v_R$.
Then, we have
\begin{equation}\label{EL4.1D}
\begin{aligned}
\Exp^v_x\Bigl[ \E^{\int_0^{t} (c_v(X_s)-\lambda^v_n)\,\D{s}}\,
&\psi^v_n(X_{t})\Ind_{\{t <\uuptau_r\wedge \uptau_n\}}\Bigr]\\
&\,\le\, \Exp^v_x\left[\E^{\int_0^{t} (\Hat{c}_v(X_s)-\lambda^v_n)\,\D{s}}\,
\psi^v_n(X_{t})\Ind_{\{t <\uuptau_r\wedge \uptau_n\}}\right]\\
&\,\le\,  \frac{\max_{\sB_n}\psi^v_n}{\min_{\sB_n}\Hat\psi^v_R}\,
\Exp^v_x\left[\E^{\int_0^{t} (\Hat{c}_v(X_s)-\lambda^v_n)\,\D{s}}\,
\Hat\psi^v_R(X_{t})\Ind_{\{t <\uuptau_r\wedge \uptau_n\}}\right]
\\
&\,\le\,  \frac{\max_{\sB_n}\psi^v_n}{\min_{\sB_n}\Hat\psi^v_R}
\E^{(\Hat\lambda^v_R-\lambda^v_n)t}\,\Hat\psi^v_R(x) \,\xrightarrow[t\to\infty]{}\, 0\,,
\end{aligned}
\end{equation}
where in the last inequality we use
\begin{equation*}
\Hat\psi^v_R(x) \,\ge\, \Exp^v_x\left[ \E^{\int_0^t
(\Hat{c}_v(X_s)-\Hat\lambda^v_R)\,\D{s}}\,
\Hat\psi^v_R(X_t)\Ind_{\{t<\uuptau_r \wedge\uptau_n\}}\right].
\end{equation*}
Therefore, decomposing the integral in
\cref{EL4.1B}, and using the monotone convergence theorem and \cref{EL4.1D}, we obtain
\begin{equation*}
\psi^v_n(x) \,\le\, \Exp^v_x\left[ \E^{\int_0^{\uuptau_r} (c_v(X_s)-\lambda^v_n)\,\D{s}}\,
\psi^v_n(X_{\uuptau_r})\Ind_{\{\uuptau_r <\uptau_n\}}\right].
\end{equation*}
This together with \cref{EL4.1C} proves \cref{EL4.1A}.
The second assertion is quite standard,
and follows from the It\^{o} formula (see \cref{L3.1}).
This completes the proof.
\end{proof}

For a diffusion as in \cref{E-sde} without the L\'evy driving term,
provided that the sequence of eigenvalues $\{\lambda_n\}$ is bounded,
Harnack's inequality enables us to
construct a principal eigenfunction on the whole space.
A standard argument then shows that the limit $\lim_{n\to\infty}\lambda_n$
cannot be a negative number.
For the model at hand, this venue does not seem possible.
However, in order
to use the function $\Lyap$ in \cref{EA1.1A} as a barrier for the sequence
of Dirichlet solutions, as done in \cref{L4.3} later in this section, we need
a lower bound of this limit.
This is provided in following lemma
which uses a weaker hypothesis than \cref{A1.1}.

\begin{lemma}\label{L4.2}
Suppose there exist a positive function $\widetilde\Lyap\in\cC^2(\Rd)$, a constant
$\widetilde{C}$, and a compact set $\widetilde\cK$, which satisfy
\begin{equation}\label{EL4.2A}
\cA\widetilde\Lyap(x,\zeta) \,\le\,\widetilde{C}\Ind_{\widetilde\cK}(x) -1\qquad
\forall\, (x,\zeta)\in\Rd\times\Act\,,
\end{equation}
with $\cA$ as in \cref{EcA}.
Then $\lim_{n\to\infty} \lambda_{\sB_n}(\cI)\ge 0$.
\end{lemma}

\begin{proof}
Since $c$ is assumed nonnegative, then in view of \cref{T2.3} it is enough
to prove the result for $c\equiv0$.
Also, without loss of generality, we assume that $\widetilde\cK\subset\sB_1$.
We argue by contradiction.
Suppose $\lambda_n=\lambda_{\sB_n}(\cI)\nearrow\lamstr<0$.
Let $\psi_n$ be the eigenfunction of $\cI$ on $\sB_n$ normalized so
that $\min_{\overline\sB_1} \psi_n = 1$.
Its existence is asserted in \cref{T2.1}.
By It\^o's formula we have
\begin{equation*}
\psi_n(x) \,=\, \inf_{v\in\Usm}\,
\Exp_x^v \Bigl[\E^{-\lambda_n(\uuptau_1\wedge\uptau_n)}\,
\psi_n(X_{\uuptau_1\wedge\uptau_n})\Bigr]\,,
\end{equation*}
and the infimum is realized at some element of $\Usm$.
Therefore, invoking \cref{L4.1}, and the inequality $\lambda_n<\lambda_*$,
which holds by \cref{C2.1}, we obtain
\begin{equation}\label{PL4.2A}
\begin{aligned}
\psi_n(x) &\,\ge\, \inf_{v\in\Usm}\,
\Exp_x^v \Bigl[\E^{-\lamstr\uuptau_1}\,\psi_n(X_{\uuptau_1})
\Ind_{\{\uuptau_1<\uptau_n\}}\Bigr]\\
&\,\ge\, \inf_{v\in\Usm}\, \Exp_x^v\bigl[\bigl(1-\lamstr\uuptau_1\bigr)
\Ind_{\{\uuptau_1<\uptau_n\}}\bigr]\,.
\end{aligned}
\end{equation}
Let $g_n(x) = \inf_{v\in\Usm}\,\Prob_x^v (\uuptau_1<\uptau_n)$.
It is clear that $g_n(x)$, being a bounded solution of a Dirichlet problem,
is in $\Sobl^{2,p}(\sB_n\cap\overline\sB_1^c)$, $p>d$, and converges to the
function $1$ uniformly on compact subsets of $\overline\sB_1^c$.
Let
\begin{equation*}
h_n(x) \,\df\, \inf_{v\in\Usm}\,\Exp_x^v\biggl[
\int_0^{\uuptau_1} g_n (X_t)\,\Ind_{\{\uuptau_1<\uptau_n\}}\,\D{t}\biggr]\,.
\end{equation*}
Using the function in \cref{EL4.2A} as a barrier, we deduce
that the family $\{h_n\}$ are solutions of Dirichlet problems,
and they are nondecreasing in $n$, and are locally bounded uniformly in $n$.
Thus, $h_n$ converges uniformly on compact subsets of $\overline\sB_1^c$ to
 $h(x)= \inf_{v\in\Usm}\,\Exp_x^v [\uuptau_1]$ on $\sB_1^c$,
which is also a solution of an exterior Dirichlet problem.
It also follows by the strong maximum principle that $h$ is positive
on $\overline\sB_1^c$.
Note that
\begin{equation}\label{PL4.2B}
\Exp_x^v\bigl[\bigl(1-\lamstr\uuptau_1\bigr)
\Ind_{\{\uuptau_1<\uptau_n\}}\bigr] \,\ge\, g_n(x) -\lamstr h_n(x)
\qquad \forall\,x\in \sB_n\setminus \overline\sB_1\,,
\end{equation}
and for all $v\in\Usm$ by construction.
Let
\begin{equation}\label{PL4.2C}
F_n(x)\,\df\, 2\wedge\bigl(g_n(x) -\lamstr h_n(x)\bigr)\,,\quad x\in\overline\sB_1^c\,.
\end{equation}
We have already shown
that $F_n$ converges uniformly on compact sets to some
continuous bounded function $F>1$ on $\overline\sB_1^c$.

On the other hand, using It\^o's formula and Fatou's lemma, we have
\begin{equation}\label{PL4.2D}
\begin{aligned}
\psi_n(x) &\,\ge\,
\inf_{v\in\Usm}\,\Exp_x^v \Bigl[\E^{-\lamstr\uptau_2}\,\psi_n(X_{\uptau_2})\Bigr]\\
&\,\ge\, \inf_{v\in\Usm}\,\Exp_x^v \bigl[\psi_n(X_{\uptau_2})\bigr]\\
&\,\ge\, \inf_{v\in\Usm}\,\Exp_x^v
\bigl[F_n(X_{\uptau_2})\bigr]
\quad\forall\, x\in\overline\sB_1\,,
\end{aligned}
\end{equation}
where in the last inequality we use \cref{PL4.2A,PL4.2B,PL4.2C}.
Consider the hitting distributions
\begin{equation*}
\beta^v_x(A) \,\df\, \Prob_x^v(X_{\uptau_2}\in A)\,,\quad x\in\overline\sB_1\,,
\ v\in\Usm\,.
\end{equation*}
We claim that the family
$\bigl\{\beta^v_x \colon x\in\overline\sB_1\,,\, v\in\Usm\bigr\}$ is tight.
Indeed, $\beta^v_x(\sB_n^c)$ is the solution of the Dirichlet problem
\begin{equation*}
\cA_v u\,=\, 0\text{\ \ in\ }\sB_1\,,\quad u \,=\,
\Ind_{\sB_n^c}\text{\ \ on\ } \sB_1^c\,.
\end{equation*}
However, by linearity, this is equivalent to the problem
\begin{equation*}
\cA_v u\,=\, - \int_{\RR^d\setminus\{0\}} \Ind_{\sB_n^c}(x+z)\nu(x,v(x),\D{z})
\text{\ \ in\ }\sB_1\,,\quad u \,=\, 0\text{\ \ on\ } \sB_1^c\,.
\end{equation*}
The claim then follows by the ABP estimate in \cref{T6.1}, since the family
$\{\nu(x,v,\cdot\,)\colon x\in\sB_1\}$ is tight by \hyperlink{A2}{(A2)}.
Therefore,
\begin{equation}\label{PL4.2E}
\inf_{x\in\sB_1}\,\inf_{v\in\Usm}\, \int_{\sB_2^c}
F(y) \beta^v_x(\D{y})\,>\,1\,,
\end{equation}
which of course also implies by tightness that the inequality in \cref{PL4.2E}
holds if the integral is restricted to an annulus $\sB_R\setminus\sB_2$ for
some $R$ sufficiently large.
Therefore, \cref{PL4.2D,PL4.2E} together with the uniform convergence of $F_n$ on compact
sets  imply that
\begin{equation*}
\liminf_{n\to\infty}\, \inf_{\sB_1}\, \psi_n \,>\, 1\,,
\end{equation*}
which contradicts the assumption that $\min_{\overline\sB_1} \psi_n = 1$.
This completes the proof.
\end{proof}

Next we prove the existence of an eigenfunction on $\Rd$.

\begin{lemma}\label{L4.3}
Grant \cref{A1.1}, and
let $\lambda^v\df\lim_{n\to\infty}\lambda^v_n$. The following hold
for $v\in\Usm$:
\begin{itemize}
\item[(a)]
There exists a positive function
$\Psi_v\in \Sobl^{2,p}(\Rd)\cap\order(\Lyap)$, $p>d$, satisfying
\begin{equation}\label{EL4.3A}
\cA_v\Psi_v + c_v\, \Psi_v \,=\, \lambda^v \Psi_v\quad \text{in\ } \Rd\,.
\end{equation}
Furthermore, $\lambda^v=\lambda_{\Rd}(\cA_v^c)$.

\item[(b)]
There exists $r_\circ\ge0$, such that for all $r>r_\circ$ we have
\begin{equation}\label{EL4.3B}
\Psi_v(x) \,=\, \Exp^v_x\left[ \E^{\int_0^{\uuptau_r} (c_v(X_s)-\lambda^v)\,\D{s}}\,
\Psi_v(X_{\uuptau_r})\right] \qquad \forall\,x\in \overline\sB^c_r\,.
\end{equation} 
\end{itemize}
\end{lemma}

\begin{proof}
We first show that $\sE^v_x(c)$ is finite.
The proof is the same under parts (a) or (b) of \cref{A1.1}, so
we work here under part (a).
Choosing $g=\widehat{C} \bigl(\min_\cK \Lyap\bigr)^{-1}\Ind_\cK$, we
write \cref{EA1.1B} as
\begin{equation*}
\cA_v\Lyap_n + (\ell-g)\Lyap_n\,\le\,\cA_v\Lyap + (\ell-g)\Lyap \,\le\, 0\,.
\end{equation*}
Let $\Lyap_n\le\Lyap$ be a sequence of increasing functions in 
$\cC^2(\Rd)\cap\cC_b(\Rd)$ such that $\Lyap_n=\Lyap$ on $\sB_n$.
Note then that
\begin{equation*}
\cA_v\Lyap_n + (\ell-g)\Lyap_n \,\le\, 0 \text{\ \ on\ } \sB_n\,.
\end{equation*}
Thus, by the It\^o's formula and \cref{L3.1}, we obtain
\begin{align*}
&\Exp^v_x\left[\E^{\int_0^{\uptau_m\wedge T} (\ell(X_s)-g(X_s))\,\D{s}}\,
\Lyap_n(X_{\uptau_m\wedge T})\right]-\Lyap(x) \\
&\mspace{100mu}\le\, \Exp^v_x\left[\int_0^{\uptau_m\wedge T}
\E^{\int_0^t (\ell(X_s)-g(X_s))\,\D{s}}
\Bigl(\cA_v\Lyap_n(X_t) + (\ell(X_t)-g(X_t))\Lyap_n(X_t)\Bigr)\,\D{t}\right] \\
&\mspace{100mu}\le\, \Exp^v_x\left[\int_0^{\uptau_m\wedge T}
\E^{\int_0^t (\ell(X_s)-g(X_s))\,\D{s}}
\Bigl(\cA_v\Lyap(X_t) + (\ell(X_t)-g(X_t))\Lyap(X_t)\Bigr)\,\D{t}\right]
\,\le\, 0
\end{align*}
for all $m\le n$.
Thus taking limits as $n\to\infty$,
and applying Fatou's lemma, we arrive at
\begin{equation*}
\Lyap(x) \,\ge\,\Exp^v_x\left[\E^{\int_0^{\uptau_m\wedge T} (\ell(X_s)-g(X_s)) \,\D{s}}\,
\Lyap(X_{\uptau_m\wedge T})\right]\qquad\forall\, m\in\NN\,.
\end{equation*}
Now let $m\to\infty$, and apply Fatou's lemma once more, to obtain
\begin{equation*}
\Lyap(x) \,\ge\, \Exp^v_x\left[\E^{\int_0^{T} (\ell(X_s)-g(X_s)) \,\D{s}}\,
\Lyap(X_{ T})\right]
\,\ge\, \biggl(\inf_\Rd\,\Lyap\biggr)\,
\Exp^v_x\left[\E^{\int_0^{T} (\ell(X_s)-g(X_s)) \,\D{s}}\right].
\end{equation*}
Taking logarithm on both sides, dividing by $T$,
and letting $T\to\infty$, we deduce that $\sE^v_x(\ell-g)<\infty$.
Since $c\in\order(\ell)$, we have $\sE^v_x(c)<\infty$.

As a consequence of the above estimate, together with
\cref{L4.1,L4.2}, and the nonnegativity of $c$, we have $\lambda^v\in[0,\infty)$.
Now choose a ball $\sB\supset \cK$ such that
$(c_v-\lambda_n)\le \ell$ (or $\gamma$)
in $\sB^c$ for all $n$ large enough.
This is possible due to \cref{L4.2}. 
Recall the definition in \cref{E-barnu},
and let $\Lg_v$ denote the `local part' of the operator $\cA_v$,
that is,
\begin{equation}\label{E-Lg}
\Lg_v u(x) \,\df\, \trace(a\grad^2 u(x)) + b_v(x)\cdot\grad u(x) -\Bar\nu u(x)\,.
\end{equation}
We scale $\psi^v_n$ in \cref{L4.1} so that it touches $\Lyap$ from below,
that is, we replace $\psi^v_n$ with $\kappa_n\psi^v_n$ where
\begin{equation*}
\kappa_n \,\df\,
\max\,\bigl\{\kappa\in(0,\infty)\colon\Lyap-\kappa\psi^v_n> 0 \text{\ in\ } \Rd\bigr\}\,.
\end{equation*}
We claim that $\psi^v_n$ can only touch $\Lyap$ in $\sB$.
Indeed, using \cref{E-Lg}, we have
\begin{equation*}
\cA_v (\Lyap-\psi^v_n) + (c_v-\lambda_n)(\Lyap-\psi^v_n)
\,\le\, 0\,\Rightarrow\,
\Lg_v (\Lyap-\psi^v_n)  -(c_v-\lambda_n)^{-} (\Lyap-\psi^v_n)
\,\le\, 0
\end{equation*}
in $\sB^c\cap\sB_n$,
and therefore, by the strong maximum principle, if $\Lyap-\psi^v_n$ vanishes somewhere
in $\sB^c\cap\sB_n$ it has to be identically zero in $\sB^c\cap\sB_n$,
which contradicts
the fact that $\psi^v_n=0$ on $\partial \sB_n$.
Thus there exists $y_n\in \sB$ such that $\Lyap(y_n)=\psi^v_n(y_n)$.
Define
\begin{equation*}
\sJ_n(x) \,\df\, \int_{\Rd} \psi^v_n(x+z) \nu(x,v,\D{z})\,.
\end{equation*}
The Foster--Lyapunov equations
in \cref{EA1.1A,EA1.1B} imply that
$x\mapsto \int_{\Rd} \Lyap(x+z) \nu(x,\D{z})$ is locally bounded.
Therefore, $\sJ_n$ is locally bounded, uniformly in $n$,
since $\psi^v_n\le\Lyap$ under the scaling above.
We write
\begin{equation*}
\Lg_v \psi^v_n + (c-\lambda_n)\psi^v_n \,=\,
-\sJ_n\quad \text{in\ } \sB_n\,.
\end{equation*}
Then by \cite[Theorem~9.20 and~9.22]{GilTru} it follows that for any domains
$D_1\Supset D\supset\sB$
there exists a constant $\kappa$ such that
\begin{equation*}
\sup_{D}\,\psi^v_n \,\le\, \kappa\,\biggl(\inf_{D}\,\psi^v_n + \norm{\sJ_n}_{L^d(D_1)}\biggr)
\,\le\, \kappa\,\biggl(\inf_{D}\,\Lyap + \norm{\sJ_n}_{L^d(D_1)}\biggr)\,.
\end{equation*}
Thus, using the standard theory of elliptic PDE \cite{GilTru},
we deduce that $\norm{\psi^v_n}_{\Sob^{2,p}(D)}$, $p>d$,
is bounded uniformly in $n$, for every fixed bounded set $D$.
Hence we can extract a subsequence $\{\psi^v_{n_k}\}$ such that 
\begin{equation*}
\psi^v_{n_k}\rightharpoonup \Psi_v\quad \text{in\ } \Sobl^{2,p}(\Rd)\,, \quad
\text{and}\quad \psi^v_{n_k}\to \Psi_v\quad \text{in\ } \cC^{1, \alpha}_{\mathrm{loc}}(\Rd)
\end{equation*}
for some $\Psi_v\in \Sobl^{2,p}(\Rd)\cap\order(\Lyap)$.
Moreover, we have 
\begin{equation*}
\cA_v\Psi_v + c_v\Psi_v \,=\, \lambda^v\Psi_v \text{\ \ a.e.\ in\ }\Rd\,.
\end{equation*}
Since $\min_\sB(\Lyap-\Psi_v)=0$ by construction and $\Lyap$ is positive,
it follows from the strong maximum principle that $\Psi_v>0$ in $\Rd$. This gives
\cref{EL4.3A}.
It is evident from \cref{E-lamD} that $\lambda^v\leq \lambda_{\Rd}(\cA_v^c)$. On the 
other hand, using \cref{E-lamD} and \cref{EL4.3A}, it follows
that $\lambda_{\Rd}(\cA_v^c)\leq \lambda^v$. Thus
$\lambda^v=\lambda_{\Rd}(\cA_v^c)$.
This completes the proof of part (a).

Next we prove part (b).
Recall the function $g$ defined in the beginning of the proof, and let
$r_\circ$ be such that $\sB_{r_\circ}\supset\sB$.
Using a similar argument as above, under \cref{A1.1}\,(b), we obtain 
\begin{equation*}
\Lyap(x) \,\ge\,
\Exp^v_x\left[\E^{\int_0^{\uuptau_r\wedge t} (\ell(X_s)-g(X_s))\,\D{s}} \right]
\,=\, \Exp^v_x\left[\E^{\int_0^{\uuptau_r\wedge t} \ell(X_s)\,\D{s}} \right]\quad
\forall\,r\ge r_\circ\,,\ \ \forall\, t\ge 0\,.
\end{equation*}
Letting $t\to\infty$, and using the fact that $\Prob^v_x(\uuptau<\infty)=1$,
by Fatou's lemma we have
\begin{equation}\label{PL4.3A}
\Lyap(x) \,\ge\, \Exp^v_x\left[\E^{\int_0^{\uuptau_r} \ell(X_s)\,\D{s}} \right]
\quad\forall\,r\ge r_\circ\,.
\end{equation}
Under \cref{A1.1}\,(a) we arrive at a similar conclusion with $\ell$ replaced by
$\gamma$.
Note that, in either case, there exists a ball $\sB$ so that
$(c-\lambda_n)\le \ell$ (or $\gamma$) in $\sB^c$ for all large $n$.
Thus \cref{PL4.3A} enables us to use the dominated convergence theorem to take limits in
\cref{EL4.1A} and obtain \cref{EL4.3B}. 
This concludes the proof.
\end{proof}

As a consequence of the
stochastic representation in \cref{EL4.3B}
we obtain the following strict monotonicity result.

\begin{lemma}\label{L4.4}
Suppose that \cref{A1.1} holds for two
cost functions
$c$ and $\Hat{c}$, such that $\Hat{c}\lneq c$.
Then we have $\lambda_{\Rd}(\cA_v^c)<\lambda_{\Rd}(\cA_v^{\Hat{c}})$
for any $v\in\Usm$.
\end{lemma}

\begin{proof}
Abusing the notation, we write $\lambda^v(c) = \lambda_{\Rd}(\cA_v^c)$.
Monotonicity implies that $\lamstr(\Hat{c})\le \lamstr(c)$.
Suppose that $\lambda^v(\Hat{c})=\lambda^v(c)$. Let
$\Psi_v$ and $\widehat\Psi_v$ be the eigenfunctions corresponding to $\cA_v^c$
and $\cA_v^{\Hat{c}}$, respectively. From \cref{L4.3} we see that the
stochastic representation formula \cref{EL4.3B} holds for $\Psi_v$ and $\widehat\Psi_v$.
Choose $\kappa>0$ such that the minimum of $\kappa\Psi_v-\widehat\Psi_v$ on
$\sB\supset\sB_{r_0}$
equals $0$, that is, $\kappa=\max_{\sB} \widehat\Psi_v (\Psi_v)^{-1}$.
Applying the stochastic representation in \cref{EL4.3B}, it then follows that
$\kappa\Psi_v\ge \widehat\Psi_v$.
Writing the difference of the two eigenvalue equations, and using \cref{E-Lg},
we obtain
\begin{equation*}
\Lg_v \bigl(\kappa\Psi_v-\widehat\Psi_v\bigr)
- \bigl(c_v-\lambda^v(c)\bigr)^-\bigl(\kappa\Psi_v-\widehat\Psi_v\bigr)
\,\le\, 0 \quad \text{in\ } \Rd\,.
\end{equation*}
Therefore, by the strong maximum principle, we must have $\kappa\widehat\Psi_v=\Psi_v$
in $\Rd$ which contradicts the fact that $\Hat{c}\lneq c$.
This completes the proof.
\end{proof}

Another consequence of the stochastic representation is uniqueness
of the principal eigenfunction.

\begin{theorem}\label{T4.2}
Grant \cref{A1.1}.
Let  $u\in \Sobl^{2,p}(\Rd)\cap\order(\Lyap)$, $p>d$,
be a positive function satisfying
\begin{equation}\label{ET4.2A}
\cA_v u+ c_vu \,=\, \lambda u \quad \text{a.e.\ in\ } \Rd
\end{equation}
for some $\lambda\ge \lambda^v$, and
\begin{equation}\label{ET4.2B}
u(x) \,=\, \Exp^v_x\left[ \E^{\int_0^{\uptau(\sB^c)} (c_v(X_s)-\lambda)\,\D{s}}\,
u\bigl(X_{\uptau(\sB^c)}\bigr)\right] \quad \forall\,x\in \sB^c\,,
\end{equation}
for some ball $\sB$.
Then we have $\lambda=\lambda^v$ and $u=\kappa\Psi_v$ for some $\kappa>0$,
with $\Psi_v$ as in \cref{EL4.3A}.
\end{theorem}

\begin{proof}
Due to strong Markov property we may assume that $\sB\supset\sB_{r_0}$.
We choose a constant $\kappa>0$ such that the
minimum of $\kappa\Psi-u$ on $\sB$ equals $0$.
By \cref{ET4.2B,EL4.3B} it then follows that $\kappa\Psi-u\ge 0$ and its minimum 
is attained in $\sB$.
As in the proof of \cref{L4.4},  we have
\begin{equation*}
\Lg (\kappa\Psi-u)
- (c-\lamstr)^-(\kappa\Psi-u) \,\le\, 0\quad \text{in\ } \Rd
\end{equation*}
by \cref{ET4.2A}.
An application of the strong maximum principle then shows that $\kappa\Psi=u$,
which in turn,
implies that $\lambda=\lamstr$.
\end{proof}

The next lemma shows that $\lamstr$ is equal to the risk-sensitive value. 

\begin{lemma}\label{L4.5}
Under \cref{A1.1}, we have $\lambda^v=\sE^v_x(c)$ for all $x\in\Rd$ and $v\in\Usm$. 
\end{lemma}

\begin{proof}
We have already shown that  $\lambda^v\le \sE^v_x(c)$ in \cref{L4.1}.
Thus we need to show the reverse inequality.
We first establish this under \cref{A1.1}\,(a).
Choose $\varepsilon>0$ small enough so that
$\gamma_\varepsilon= \norm{c}_\infty + \varepsilon<\gamma$, and define
\begin{equation*}
\Breve{c}_n(x) \,\df\, c_v(x)\Ind_{\sB_n}(x) + \gamma_\varepsilon \Ind_{\sB^c_n}(x)\,.
\end{equation*}
We have $\lambda_{\Rd}(\Breve{c}_n)<\gamma_\varepsilon$ by \cref{L4.4}.
Let $(\Breve\Psi_n, \Breve\lambda_n)$ be the eigenpair satisfying \cref{L4.3}\,(a)
with the cost function $\Breve{c}_n$.
Note that $\Breve{c}_n\ge\Breve\lambda_n$ in $\sB_n^c$.
Using It\^{o}'s formula it is straightforward to verify that
$\Breve\Psi_n(x)\ge \min_{\sB_n}\Breve\Psi_n$ (see for instance, \cref{L4.3}\,(b)).
Again, applying \cref{L3.1} together with Fatou's lemma we obtain
\begin{equation*}
\biggl(\min_{\sB_n}\,\Breve\Psi_n\biggr)
\Exp^v_x\Bigl[\E^{\int_0^T (\Breve{c}_n(X_s)-\Breve\lambda_n) \,\D{s}}\Bigr]
\,\le\, \Exp^v_x\left[\E^{\int_0^T (\Breve{c}_n(X_s)-\Breve\lambda_n) \,\D{s}}\,
\Breve\Psi_n(X_T)\right] \,\le\, \Breve\Psi_n(x)\,.
\end{equation*}
Now taking logarithm on both sides, dividing by $T$, and letting $T\to\infty$,
we obtain $\Breve\lambda_n\ge \sE^v_x(c_n)$.
In particular, we have
\begin{equation*}
\Breve\lambda_n \,=\, \sE^v_x(\lambda^v\Breve{c}_n) \,\ge\, 0\quad \forall \, n\in\NN\,.
\end{equation*}
Also note that
$\Breve\lambda\df\lim_{n\to\infty} \Breve\lambda_n \ge \sE^v_x(c)\ge \lambda^v(c)$.

In order to complete the proof, it remains to show that $\Breve\lambda=\lambda^v$.
As in the proof of \cref{L4.3}, we can find $r_\circ>0$
such that $\Breve{\Psi}_n\le \Lyap$ and it touches
$\Lyap$ at some point in $\sB_{r_\circ}$.
We can then use \cref{EA1.1A} as a barrier, and pass to the limit
to obtain some positive $\Breve\Psi\in\Sobl^{2,p}(\Rd)$ which satisfies
\begin{equation*}
\cA_v \Breve\Psi + c_v\Breve\Psi \,=\, \Breve\lambda\Breve\Psi\quad\text{on\ }\Rd\,.
\end{equation*}
By \cref{L4.3}\,(b) we have
\begin{equation}\label{PL4.5A}
\Breve\Psi_n(x) \,=\, \Exp^v_x\left[ \E^{\int_0^{\uuptau_r}
(\Breve{c}_n(X_s)-\Breve\lambda_n)\,\D{s}}\,
\Breve\Psi_n(X_{\uuptau_r})\right] \quad \forall\, x\in \sB^c_r\,,
\end{equation}
for some $r>0$.
We can then use \cref{PL4.3A} and dominated convergence to
take limits in \cref{PL4.5A} as $n\to\infty$ to obtain
\begin{equation*}
\Breve\Psi(x) \,=\, \Exp^v_x\left[ \E^{\int_0^{\uuptau_r} (c_v(X_s)-\Breve\lambda)\,\D{s}}\,
\Breve\Psi(X_{\uuptau_r})\right] \quad \forall\, x\in \sB^c_r\,.
\end{equation*}
Combining this with \cref{T4.2} completes the proof.

Next, consider \cref{A1.1}\,(b).
Here, we define
\begin{equation*}
\Breve{c}_n(x) \,\df\, c_v(x) + \frac{1}{2}\bigl(\ell(x)-c_v(x)\bigr)^+
\Ind_{\sB^c_n}(x)\,,
\end{equation*}
and let $(\Breve\Psi_n,\Breve\lambda_n)$ be the eigenpair associated with this running
cost.
Then we can repeat the above argument to first deduce that
$\Breve\lambda \ge \sE^v_x(c)\ge \lambda^v$,
and then establish that $\Breve\lambda=\lambda^v$.
This completes the proof.
\end{proof}

Now are now ready to complete the proof of \cref{T4.1}.

\begin{proof}[Proof of \cref{T4.1}]
Part (a) follows from \cref{L4.3}. Part (b) follows from \cref{L4.3}(a) and \cref{L4.5}.

It remains to prove part (c).
Let $u\in\Sobl^{2,d}(\Rd)\cap\order(\Lyap)$ be a positive solution to
\begin{equation*}
\cA_v u+ c_v u \,=\, \lambda^v u \quad \text{in\ } \Rd\,.
\end{equation*}
Applying \cref{L3.1} and Fatou's lemma it is easily seen that
\begin{equation*}
u(x) \,\ge\, \Exp^v_x\left[ \E^{\int_0^{\uptau(\sB^c)} (c_v(X_s)-\lambda^v)\,\D{s}}\,
u\bigl(X_{\uptau(\sB^c)}\bigr)\right] \quad \forall\,x\in \sB^c\,,
\end{equation*}
for any ball $\sB$.
As done earlier, choose $\kappa>0$ such that 
the minimum of $\kappa u-\Psi_v$ in $\overline\sB$ equals $0$.
Applying the arguments in the proof of \cref{T4.2} we deduce that $\kappa u=\Psi_v$.
This completes the proof.
\end{proof}

\subsection{Proof of \texorpdfstring{\cref{T1.1}}{}}\label{S4.2}
In this section we present the proof of \cref{T1.1}.

\begin{proof}
Let $(V_n,\lambda_n)$ be the principal Dirichlet eigenpair of the operator $\cI$
in $\sB_n$.
Following the arguments of \cref{L4.3} and using \cref{EA1.1C}, we can show that there exists
some positive function $V\in\Sobl^{2,p}(\Rd)\cap\order(\Lyap)$, $p>d$,
satisfying
\begin{equation}\label{PT1.1A}
\cI V \,=\, \lamstr V\quad \text{a.e.\ in\ } \Rd\,,\quad\text{and\ } V(0)=1\,,
\end{equation}
where $V$ is obtained as a subsequential limit of $V_n$ in $\Sobl^{2,p}(\Rd)$,
and $\lamstr=\lim_{n\to\infty}\lambda_n$.
This completes the proof of part (a).

Part (b) follows from a standard measurable selection theorem, for example,
Filippov’s implicit function theorem in
\cite[Theorem 18.17]{ali-bor}.

Turning to part (c), we first show that a control satisfying \cref{ET1.1B} is optimal.
By \cref{L3.1} we can employ It\^o's formula, and the
method of proof of \cite[Lemma~2.3\,(i)]{AB18} to assert
that $\lambda_n \le \sE_x(c,Z)$ for all $x\in\sB_n$, and $Z\in\Uadm$.
Therefore, in view of \cref{T4.1}, we obtain
\begin{equation}\label{PT1.1B}
\lamstr\,\le\, \sE^* \,\le\, \sE^v \,=\, \lambda^v\qquad\forall\,v\in\Usm\,,
\end{equation}
where we use the definitions in \cref{Enot4.1}.
Let $\bUsm\subset\Usm$ denote the set of Markov controls satisfying \cref{ET1.1B}.
If $v\in\bUsm$, or equivalently, if $\cA^c_v V = \lamstr V$, then
\cref{T4.1} and \cref{PT1.1B} imply that $\Psi_v=V$ and $\lamstr=\lambda^v$.
Thus we have shown that
\begin{equation}\label{PT1.1C}
V\,=\,\Psi_v\,,\quad\text{and\ \ } \lamstr=\sE^*=\lambda^v
\qquad\forall\,v\in\bUsm\,.
\end{equation}
It is also clear that $\lamstr=\lambda_\Rd(\cI)$, with the second as
defined in \cref{E-lamD}.

Now, let $\Usms$ denote the class of optimal stationary Markov controls,
and choose an arbitrary $v\in\Usms$.
Then $\cI V_n = \lambda_n V_n$ implies that $\cA^c_v V_n \ge \lambda_n V_n$,
which in turn implies that
\begin{equation}\label{PT1.1D}
\begin{aligned}
V_n(x) &\,\le\, \Exp_x^v \Bigl[\E^{\int_{0}^{\uuptau_r}
[c_v(X_s)-\lambda_n]\,\D{s}}\,
V_n(X_{\uuptau_r})\,\Ind_{\{\uuptau_r<T\wedge\uptau_n\}}\Bigr]\\
&\mspace{100mu} +\Exp_x^v \Bigl[\E^{\int_{0}^{T}[c_v(X_s)-\lambda_n]\,\D{s}}\,
V_n(X_{T})\,
\Ind_{\{T<\uuptau_r\wedge\uptau_n\}}\Bigr]\quad\forall\,T>0\,,
\end{aligned}
\end{equation}
$r\in(0,n)$, 
and $x\in \sB_n\setminus\sB_r$.
Choose $r$, and $n$ large enough so that
$\norm{c}_\infty-\lambda_n< \gamma$ and $\cK\subset\sB_{r}$.
We claim that the last term in \cref{PT1.1D} tends to $0$ as $T\to\infty$.
Indeed, since $V_n\leq \Lyap$, we have
\begin{align*}
\Exp_x^v \Bigl[\E^{\int_{0}^{T}[c_v(X_s)-\lambda_n]\,\D{s}}\,
V_n(X_{T})\,
\Ind_{\{T<\uuptau_r\wedge\uptau_n\}}\Bigr] & \,\le\,
\E^{(\norm{c}_\infty-\lambda_n-\gamma)T}\, \Exp_x^v \Bigl[\E^{\gamma T }\,
\Lyap(X_{T})\,
\Ind_{\{T<\uuptau_r\wedge\uptau_n\}}\Bigr]
\\
&\,\le\, \E^{(\norm{c}_\infty-\lambda_n-\gamma)T}\, \Lyap(x)
\,\xrightarrow[T\to\infty]{}\, 0\,,
\end{align*}
where the second inequality follows by \cref{EA1.1A}.
Same conclusion holds under \cref{A1.1}(b).
Thus, first taking limits in \cref{PT1.1D} as $T\to\infty$, using monotone
convergence for the first term, and then
employing \cref{A1.1} and dominated convergence to take limits as $n\to\infty$,
we obtain
\begin{equation}\label{PT1.1E}
V(x) \,\le\, \Exp_x^v \Bigl[\E^{\int_{0}^{\uuptau_r}
[c_v(X_s)-\lamstr]\,\D{s}}\,V(X_{\uuptau_r})\Bigr]\,.
\end{equation}
Using \cref{PT1.1E} together with $\cA^c_v V\ge\lamstr V$ and
$\cA^c_v \Psi_v=\lambda^v \Psi_v$ from \cref{T4.1}, and the fact that
$\lambda^v=\lamstr$ by the optimality of $v$ and \cref{PT1.1C},
it follows as in the proof of \cref{T4.2} that $\Psi_v=V$.
Thus we have shown
\begin{equation}\label{PT1.1F}
\Usms\,=\,\bUsm\,.
\end{equation}

\Cref{PT1.1C,PT1.1F} show that $\lamstr=\sE^*$, and any solution $V$ of
\cref{PT1.1A} equals $\Psi_v$ for any optimal stationary Markov control $v\in\Usm$.
This of course implies uniqueness of the solution and the verification
of optimality result in the theorem, and completes the proof.
\end{proof}

\section{A risk-sensitive maximization problem}\label{S-max}

In this section we study a risk-sensitive (reward) maximization problem.
In addition to \hyperlink{A1}{(A1)}--\hyperlink{A3}{(A3)}, throughout
this section
we assume the following.

\smallskip
\begin{enumerate}[(A4)]
\item[\hypertarget{A4}{(A4)}]
For some constant $C_0$ we have
\begin{equation}\label{EA4A}
\sup_{\zeta\in\Act}\, \langle b_\circ(x, \zeta),x\rangle^{-}
\,\le\, C_0\, (1+\abs{x}^2)\quad \forall\, x\in\Rd\,.
\end{equation}
In addition, $\nu$ satisfies
\begin{equation}\label{EA4B}
\sup_{(x, \zeta)\in\Rd\times\Act}\,
\int_{\Rd} \frac{\abs{x}^2}{1+\abs{x+z}^2}\,\nu(x,\zeta,\D{z})\,<\,\infty\,.
\end{equation}
\end{enumerate}
\smallskip

We note that \cref{EA4B} holds if for some $\theta\in (0,1)$ we have
$|g(x, \zeta,\xi)|\leq \theta |x|$ for all
$\zeta\in\Act, \xi\in\RR^m\setminus\{0\}$, and
all large enough $|x|$.
In this section, $c\colon\Rd\times\Act\to\RR$
is a continuous function which is bounded from above,
representing a running reward.
With $\sE_x(c, Z)$ as in \cref{D1.1},
the optimal value for the maximization problem is defined as
\begin{equation}\label{E-max2}
\Hat\sE^* \,\df\, \sup_{x\in\Rd}\, \sup_{Z\in\Uadm}\, \sE_x(c,Z)\,,
\end{equation}
respectively.
For this maximization problem, the operator takes the form
\begin{equation*}
\cIm f(x) \,\df\, \trace\bigl(a(x)\grad^2 f(x)\bigr) 
+ \max_{\zeta\in\Act}\,\bigl\{I[f,x, \zeta] + b(x,\zeta)\cdot \grad f(x) + c(x,\zeta) f(x)\bigl\}
\end{equation*}
for $f\in \cC^2(\Rd)\cap\cC_b(\Rd)$. By \cref{T3.1} there exists a unique
$w_n \in \cC_{b}(\Rd)\cap\Sobl^{2,p}(\sB_n)$, $p>d$,
satisfying 
\begin{equation}\label{E-max3}
\begin{split}
\cIm w_n &\,=\, \varrho_n \, w_n \quad \text{in\ } \sB_n\,,\\
w_n &\,=\,0 \quad \text{in\ } \sB_n^c\,,\\
w_n &\,>\, 0 \quad \text{in\ } \sB_n\,, \quad w_n(0)= 1.
\end{split}
\end{equation}
Furthermore, $\varrho_n< \varrho_{n+1}$ for all $n$.
We assume the following \textit{near monotone} condition.

\smallskip
\noindent
\textbf{\hypertarget{H}{(H)}}
The running reward function is bounded above in $\Rd$, and
\begin{equation*}
\lim_{n\to\infty}\,\varrho_n \,>\, 
\underline{C}\,\df\,\lim_{r\to\infty}\,
\sup_{(x,\zeta)\in \sB^c_r\times\Act} c(x,\zeta)\,.
\end{equation*}

\begin{remark}
Hypothesis \hyperlink{H}{(H)} implies
that the process under an optimal control cannot be transient.
This is somewhat necessary for the  risk-sensitive value
 and the principal eigenvalue of the operator $\cIm$ in $\Rd$ to be equal.
Even for local operators, that is, with $\nu=0$,
it is known from \cite[Example~3.1]{ABS19}
that  the principal eigenvalue can be strictly smaller than the risk-sensitive value,
even for uncontrolled problems.
\end{remark}

Our main result of this section is the following.

\begin{theorem}\label{T5.1}
Grant \hyperlink{A1}{\ttup{A1}}--\hyperlink{A4}{\ttup{A4}} and
\hyperlink{H}{\ttup{H}}. Then the following hold.
\begin{itemize}
\item[(a)]
$\Hat\sE^*= \lim_{n\to\infty} \varrho_n$.

\item[(b)]
There exists a unique positive
$\Phi_*\in \cC_b(\Rd)\cap\Sobl^{2,p}(\Rd)$, $p>d$, satisfying
\begin{equation}\label{ET5.1A}
\cIm \Phi_* \,=\, \Hat\sE^*\Phi_*\quad \text{in\ }\Rd\,,
\quad \text{and}\quad \Phi_*(0)=1\,.
\end{equation}

\item[(c)]
A stationary Markov control $v$ is optimal if and only if 
\begin{equation}\label{ET5.1B}
I_v[\Phi_*, x] + b_v(x)\cdot \grad \Phi_*(x) + c_v(x) \Phi_*(x) \,=\,
\max_{\zeta\in\Act}\,\bigl\{I[\Phi_*, x, \zeta] + b(x,\zeta)\cdot \grad \Phi_*(x)
+ c(x,\zeta) \Phi_*(x)\bigl\}
\end{equation}
almost everywhere in $\Rd$.
\end{itemize}
\end{theorem}

The remaining part of this section is devoted to the proof of \cref{T5.1}
which requires the results in \cref{L5.1,L5.2,L5.3} which follow.
Let us begin with the following estimate on the hitting time probabilities.

\begin{lemma}\label{L5.1}
For any $T>0$ and $r>0$, we have
\begin{equation*}
\lim_{R\to\infty}\, \sup_{x\in\sB_R^c}\, \sup_{Z\in\Uadm}\,
\Prob^Z_x(\uuptau_r< T)\,=\,0\,.
\end{equation*}
\end{lemma}

\begin{proof}
With no loss of generality we assume that $r=1$.
Let $f(x)=(1+\abs{x}^2)^{-\nicefrac{1}{2}}$.
Applying It\^{o}'s formula to \cref{E1.1}, and using the definition
in \cref{EcA}, we see that
\begin{align}\label{PL5.1A}
\Exp^Z_x[f(X_t)] \,=\, f(x) + \Exp_x^Z
\left[\int_0^t \cA f(X_s, Z_s)\,\D{s}\right],\quad t\ge0\,.
\end{align}
Using the growth condition of $a$ and $b$ (see \cref{EA3,EA4A})
it is easily seen that 
\begin{equation*}
\babs{\trace\bigl(a(x)\grad^2 f(x)\bigr)} + \max_{\zeta\in\Act}\,
b(x, \zeta)\cdot \grad f(x)\,\le\, \kappa f(x)\qquad \forall\,
(x,\zeta)\in\Rd\times\Act\,,
\end{equation*}
for some constant $\kappa$.
On the other hand, \cref{EA4B} implies that
\begin{equation*}
\babs{I[f, x]}\,\le\, \kappa_1 f(x)\qquad \forall\, x\in\Rd\,,
\end{equation*}
for some constant $\kappa_1$,
Thus using Gr\"{o}nwall's inequality in \cref{PL5.1A},
it follows that
\begin{equation}\label{PL5.1B}
\sup_{Z\in\Uadm}\, \sup_{0\le t\le T}\Exp^Z_x[f(X_t)]\,\le\, \kappa_2 f(x)
\qquad \forall\, x\in\Rd\,,
\end{equation}
where the constant $\kappa_2$ depends on $T$ but not on $x$.
Again, using It\^{o}'s formula, we note that
\begin{equation}\label{PL5.1C}
\begin{aligned}
f(X_t) &\,=\, f(x) + \left[\int_0^t \trace\bigl(a(X_s)\grad^2 f(X_s)\bigr)
+ I[f,X_s, Z_s] + b(X_s, Z_s)\cdot \grad f(X_s)\,\D{s}\right]\\
&\mspace{50mu}+ \int_0^t \int_{\RR^m\setminus\{0\}}
\bigl(f(X_{s-}+ g(X_{s-}, \xi))-f(X_{s-})\bigr)
\bigl(\widetilde\cN(\D{t}, \D{\xi})-\Pi(\D{\xi})\,\D{t}\bigr)\\
&\mspace{150mu} + \int_0^t \grad f(X_s) \upsigma(X_s)\,\D{W_s}\,.
\end{aligned}
\end{equation}
By Doob's martingale inequality and \cref{PL5.1B}, we obtain
\begin{align*}
\Exp_x^Z\left[\sup_{t\in [0, T]}
\left|\int_0^t \grad f(X_s) \upsigma(X_s)\,\D{W_s}\right|\right]
&\,\le\,  \Exp_x^Z\left[\sup_{t\in [0, T]} \left|\int_0^t \grad f(X_s)
\upsigma(X_s)\,\D{W_s}\right|^2\right]^{\nicefrac{1}{2}}\\
&\,\le\, \sqrt{2} \Exp_x^Z\left[ \int_0^T |\grad f(X_s)|^2
|\upsigma(X_s)|^2\,\D{s} \right]^{\nicefrac{1}{2}} \\
&\,\le\, \kappa_3 \sqrt{f(x)}
\end{align*}
for some constant $\kappa_3$ depending only on $T$.
Similarly, we also get
\begin{align*}
\Exp^Z_x\left[ \sup_{t\in[0, T]}\,\biggl|\int_0^t \int_{\RR^m\setminus\{0\}}
\bigl(f\bigl(X_{s-}+ g(X_{s-},Z_s, \xi)\bigr)-f(X_{s-})\bigr)
\bigl(\widetilde\cN(\D{t}, \D{\xi})-\Pi(\D{\xi})\,\D{t}\bigr)\biggr|\right]
\,\le\, \kappa_3 \sqrt{f(x)}\,,
\end{align*}
 using the same constant $\kappa_3$, without loss of generality.
Using these estimates in \cref{PL5.1C} and applying Gr{o}nwall's inequality,
 we have
\begin{equation*}
\sup_{Z\in\Uadm}\, \Exp^Z_x\biggl[\sup_{0\le t\le T}\, f(X_t)\biggr]
\,\le\, \kappa_4 \sqrt{f(x)}\qquad \forall\, x\in\Rd\,,
\end{equation*}
for some constant $\kappa_4$. Thus
\begin{equation*}
\sup_{Z\in\Uadm}\, \Prob^Z_x(\uuptau_1< T) \,=\,
\sup_{Z\in\Uadm}\, \Prob^Z_x \biggl(\inf_{t\in [0, T]} \sqrt{1+\abs{X_t}^2}
< \sqrt{2}\biggr)
\,\le\, \sqrt{2} \kappa_4 \sqrt{f(x)}\,,
\end{equation*}
and the result follows by letting $\abs{x}\to\infty$.
\end{proof}

\begin{remark}
 Assumption \hyperlink{A4}{(A4)} is crucial for \cref{L5.1}.
Consider the case where $\nu(x,\cdot\,)=\delta_{-x}$, that is,
a Dirac mass at $-x$,
and for simplicity let $a$ be the identity matrix, and $b_\circ(x)=-x$.
Then \cref{EA4B} does not hold.
Let $f(x) \df \E^{-1} - \E^{-\abs{x}}$.
An easy calculation shows that $\cA f(x) = -1$ for $\abs{x}>1$.
This implies that $\Exp_x[\uuptau_1]\le \E^{-1}$,
and therefore, $\Prob_x(\uuptau_1< 2)\ge1-\frac{1}{2\E}$ for all $x\in \sB_1^c$,
thus violating the assertions of \cref{L5.1}.
\end{remark}

We next establish the existence of a principal eigenfunction on the whole space. 

\begin{lemma}\label{L5.2}
Suppose that \hyperlink{H}{\ttup{H}} holds.
Then there exists a bounded, positive solution
$\Phi_*\in \cC(\Rd)\cap\Sobl^{2,p}(\Rd)$
to the equation
\begin{equation}\label{EL5.2A}
\cIm \Phi_* \,=\, \varrho_*\, \Phi_*\quad \text{in\ } \Rd\,,
\end{equation}
where $\varrho_*=\lim_{n\to\infty}\varrho_n$.
In addition, $\lim_{\abs{x}\to\infty}\Phi_*(x)=0$, and there exists $r_\circ>0$ such that
for any measurable selector $v$ we have
\begin{equation}\label{EL5.2B}
\Phi_*(x) \,=\,
\Exp_x^v\left[\E^{\int_0^{\uuptau} (c_v(X_s)-\varrho_*)\,\D{s}}\,
\Phi_*(X_{\uuptau})\Ind_{\{\uuptau<\infty\}}\right],
\quad x\in \sB^c_{r_\circ}\,,
\end{equation}
with $\uuptau\equiv\uptau(\sB^c_{r_\circ})$.
\end{lemma}

\begin{proof}
Let $v_n$ be a measurable selector from $\cIm w_n = \varrho_n w_n$ in
\cref{E-max3}, that is, 
\begin{equation*}
\cA_{v_n} w_n(x)
+ c_{v_n}(x) w_n(x) \,=\, \varrho_n w_n(x)\quad \text{a.e.\ }x\in\sB_n\,.
\end{equation*}
Choose $\delta>0$ and $r_\circ>0$ satisfying
$\varrho_n-\max_{\zeta\in\Act} c(x,\zeta)\ge \delta$ for all $x\in \sB^c_{r_\circ}$,
and for all $n$ sufficiently large.
This is possible due to \hyperlink{H}{(H)}.
For the rest of the proof we set $\uuptau\equiv\uuptau_{r_\circ}$.
Using \cref{T2.1} and \cref{L4.1}, it follows that 
\begin{equation}\label{PL5.2A}
w_n(x) \,=\, \Exp^{v_n}_x\left[ \E^{\int_0^{\uuptau}
(c_{v_n}(X_s)-\varrho_n)\,\D{s}}\,
w_n(X_{\uuptau})\Ind_{\{\uuptau<\uptau_n\}}\right]
\quad \forall\,x\in \sB_n\setminus \overline\sB_{r_\circ}\,,\ \forall\,n> r_\circ\,.
\end{equation}
This of course, implies that $\sup w_n= \sup_{\sB_{r_\circ}} w_n$.
Let
\begin{equation*}
\widetilde{w}_n \,=\, \frac{1}{\sup_{\sB_{r_\circ}} w_n} w_n\,.
\end{equation*}
Thus $\widetilde{w}_n\le 1$, and it attains its maximum in the ball $\overline{\sB}_{r_\circ}$.
Thus we can apply the argument in \cref{L4.3} to extract a subsequence of
$\widetilde{w}_n$ that  converges to $\Phi_*$ in $\Sobl^{2,p}(\Rd)$, $p>d$, which satisfies
\begin{equation*}
\cIm \Phi_* \,=\, \varrho_*\, \Phi_*\quad \text{in\ }\Rd\,.
\end{equation*}
This establishes \cref{EL5.2A}.

From \cref{PL5.2A} we see that for any $x\in \sB^c_{r_\circ}$ we have
\begin{equation*}
\widetilde{w}_n (x) \,\le\,
\Exp^{v_n}_x\left[\E^{-\delta \uuptau}\, \Ind_{\{\uuptau<\uptau_n\}}\right].
\end{equation*}
Thus for any $T>0$ we have
\begin{equation*}
\widetilde{w}_n (x) \,\le\, \E^{-\delta T}  \sup_{Z\in\Uadm}\, \Prob^Z_x (\uuptau\ge T)
+  \sup_{Z\in\Uadm}\,\Prob^Z_x (\uuptau< T)\,.
\end{equation*}
Hence, by \cref{L5.1},
for any given $\varepsilon>0$, we can choose $T$ and $R$ large enough to satisfy
$\widetilde{w}_n(x) < \varepsilon$ for all $x\in \sB^c_R$.
This shows that $\lim_{\abs{x}\to\infty}\Phi_*(x)=0$.
To prove \cref{EL5.2B} we choose any $R>r_\circ$, and applying
It\^{o}'s formula, we obtain
\begin{equation}\label{PL5.2B}
\Phi_*(x) \,=\, \Exp_x^v\left[\E^{\int_0^{\uptau_R\wedge\uuptau}
(c_v(X_s)-\varrho_*)\,\D{s}}\,
 \Phi_*(X_{\uptau_R\wedge\uuptau})\right].
\end{equation}
Now we see that
\begin{align*}
\Exp_x^v\left[\E^{\int_0^{\uptau_R} (c_v(X_s)-\varrho_*)\,\D{s}}\,
\Phi_*(X_{\uptau_R})\Ind_{\{\uptau_R< \uuptau\}}\right]
\,\le\, \sup_{\abs{x}\geq R}\, \Phi_*(x) \,\xrightarrow[R\to\infty]{}\, 0\,.
\end{align*}
Thus \cref{EL5.2B} follows from \cref{PL5.2B} and
the monotone convergence theorem.
\end{proof}

In the next lemma, we show that  $\varrho_*$ is indeed the optimal value.

\begin{lemma}\label{L5.3}
Under the hypothesis of \cref{T5.1} we have that $\Hat\sE^*=\varrho_*$.
In addition, any measurable selector from \cref{ET5.1A} is an optimal
stationary Markov control.
\end{lemma}

\begin{proof}
Let $v$ be any measurable selector.
Then applying It\^{o}'s formula to \cref{EL5.2A}, and applying the dominated
convergence theorem, using also the fact that $\Phi_*\le1$ as
normalized in the proof of \cref{L5.2}, we obtain
\begin{equation*}
\Phi_*(x) \,=\,
\Exp^v_x\left[\E^{\int_0^T (c_v(X_s)-\varrho_*)\,\D{s}}\, \Phi_*(X_T)\right]
\,\le\, \Exp^v_x\left[\E^{\int_0^T (c_v(X_s)-\varrho_*)\,\D{s}} \right].
\end{equation*}
Thus, taking logarithms on both sides, dividing by $T$, and letting $T\to\infty$,
we have
\begin{equation}\label{PL5.3A}
\varrho_* \,\le\, \sE_x(c, v)\,\le\, \Hat\sE^*\qquad \forall\,x\in\Rd.
\end{equation}
To show the reverse inequality, let $\delta>0$ be given.
Consider a smooth nonnegative cut-off function $\chi$ satisfying
$\chi = 0$ in $\sB_{r_\circ}$, and $\chi = 1$ in $\sB^c_{2r_\circ}$,
with $r_\circ$ as in the proof of \cref{L5.2}.
Select $\varepsilon>0$ small enough so that
\begin{equation*}
\varepsilon\bigl (\cIm\chi - \varrho_*\chi\bigr) \,\le\, \delta \Phi_*
\quad\text{on\ }\Rd\,.
\end{equation*}
This is possible because $\chi$ equals its maximum in $\sB^c_{2r_\circ}$,
and thus $I[\chi, x,\zeta]\leq 0$ in $\sB^c_{2r_\circ}$ for $\zeta\in\Act$.
Therefore, $\phi\df\Phi_* + \varepsilon \chi$ satisfies
\begin{equation}\label{PL5.3C}
\cIm \phi(x) - (\varrho_*+\delta)\phi(x)
\,\le\, (\cIm-\varrho_*)\Phi_*(x) +\varepsilon\,(\cIm-\varrho_*)\chi(x)
- \delta\,\phi(x)\,\le\,0\quad\forall\,x\in\Rd\,.
\end{equation}
We have $\inf_{\Rd}\phi>0$ by definition.
Now we consider an admissible control $Z$ and apply It\^{o}'s formula to \cref{PL5.3C}
to obtain
\begin{align*}
\phi(x) \,\ge\, \Exp^Z_x\left[\E^{\int_0^T (c(X_s, Z_s)-\varrho_*-\delta)\,\D{s}}
\,\phi(X_T)\right]
\,\ge\,
\Bigl(\inf_{\Rd}\phi\Bigr)\,
\Exp^Z_x\left[\E^{\int_0^T (c(X_s, Z_s)-\varrho_*-\delta)\,\D{s}} \right].
\end{align*}
Take logarithms on both sides, divide by $T$, and let $T\to\infty$, to deduce that 
$\varrho_*+\delta \ge \sE_x(c, Z)$.
Since $Z$ and $\delta$ are  arbitrary, it follows that $\varrho_*\ge \Hat\sE^*$.
Thus the proof follows from \cref{PL5.3A}.
\end{proof}

We are ready to prove \cref{T5.1}.
\begin{proof}[Proof of \cref{T5.1}]
Part (a) follows from \cref{L5.2,L5.3}.
Existence of $\Phi_*$ follows from \cref{L5.2}.
To show uniqueness, consider a positive 
$u\in \Sobl^{2,p}(\Rd)$, $p>d$, satisfying
\begin{equation}\label{PT5.1A}
\cIm u \,=\, \Hat\sE^* V\quad \text{in\ }\Rd\,.
\end{equation}
Let $v$ be any measurable selector from \cref{ET5.1A}.
It follows from \cref{PT5.1A} that
\begin{equation}\label{PT5.1B}
\cA_v u(x) + c_v(x) u(x) \,\le\, \Hat\sE^* u(x)\,.
\end{equation}
An application of It\^{o}'s formula together with a Fatou's lemma gives us
\begin{equation}\label{PT5.1C}
u(x) \,\ge\, \Exp_x^v\left[\E^{\int_0^{\uuptau}
(c_v(X_s)-\Hat\sE^*)\,\D{s}}\, u(X_{\uuptau})
\Ind_{\{\uuptau<\infty\}}\right], \quad x\in \sB^c_{r_\circ}\,,
\end{equation}
where $r_\circ$ and $\uuptau=\uuptau_{r_\circ}$ are as in \cref{L5.2}.
Let 
\begin{equation*}
\kappa \,=\, \min_{\Bar{\sB}_{r_\circ}}\, \frac{u}{\Phi_*}\,.
\end{equation*}
Using \cref{EL5.2B,PT5.1C}, we deduce that $u\ge \kappa \Phi_*$ in $\Rd$,
 and that $u-\kappa\Phi_*$ equals $0$ at some point in $\Bar{\sB}_{r_\circ}$.
Let $f=u-\kappa\Phi_*$.
Using \cref{PT5.1B}, we obtain
\begin{equation*}
\cA_v f(x) - (c_v(x)-\Hat\sE^*)^- f(x)\,\le\, 0
\quad \text{in\ }\Rd\,.
\end{equation*}
It then follows by the strong maximum principle that $f=0$, or equivalently,
that $u=\kappa\Phi_*$.
This proves part (b).

We continue with part (c).
Optimality of any measurable selector of \cref{ET5.1B} follows from \cref{L5.3}.
Let $v$ be an optimal stationary Markov control, that is, 
$\sE_x(c, v)=\Hat\sE^*$ for all $x$.
Recall the linear operator $\cA^c_v$ defined in the beginning of \cref{S4.1}.
Let $\lambda_n(\cA^c_v)$ denote the Dirichlet eigenvalue
on $\sB_n$, and $\Hat\lambda$ its limit as $n\to\infty$.
If  $\Hat\lambda(\cA^c_v)>\underline{C}$, with $\underline{C}$ as defined
in \hyperlink{H}{(H)}, then using the arguments in the proof of \cref{L5.2},
there exists a bounded, positive function $\Phi_v\in\Sobl^{2,p}$, for any $p>1$,
satisfying $\cA^c_v\Phi_v = \Hat\lambda(\cA^c_v)\Phi_v$.
In addition the proof of \cref{L5.3} shows that
$\Hat\lambda(\cA^c_v)= \sE_x(c, v)=\Hat\sE^*$.
Furthermore, the stochastic representation \cref{EL5.2B}
also holds for $\Phi_v$. Thus we can apply the argument used in the proof
of part (b) to conclude that $\Phi_v=\kappa\Phi_*$
for some positive constant $\kappa$. Thus $v$
must satisfy \cref{ET5.1B}.

It remains to show that $\Hat\lambda(\cA^c_v)>\underline{C}$ for any optimal control $v$.
Assume the contrary that is, $\Hat\lambda(\cA^c_v)\leq\underline{C}$.
Let $f_n(t)\df \lambda_n(\cA^c_v+t\Ind_\sB)$, with $\sB$
the unit ball in $\Rd$.
Each function $f_n$ is convex and increasing by \cref{T2.3},
and the sequence $\{f_n\}$ is monotone (\cref{C2.1}) and pointwise bounded.
Thus, by convexity, $\{f_n\}$ is Lipschitz equicontinuous on any compact interval.
It follows that the limit $\Hat\lambda(\cA^c_v+t\Ind_\sB)$ is continuous in $t$.
It is also clear that the range of $t\mapsto \Hat\lambda(\cA^c_v+t\Ind_\sB)$,
includes $\bigl[\Hat\lambda(\cA^c_v),\infty)$, since
$\lambda_1(\cA^c_v+t\Ind_\sB)= t+ \lambda_1(\cA^c_v)$.
Thus there exists $t_\circ>0$ such that
$\underline{C} < \Hat\lambda(\cA^c_v+t_\circ\Ind_\sB)<\Hat\sE^*$.
We use \cref{L5.2} to construct a bounded positive solution $u$
to
\begin{equation*}
(\cA^c_v+t_\circ\Ind_\sB) u \,=\, 
\Hat\lambda(\cA^c_v+t_\circ\Ind_\sB)\, u\quad \text{in\ } \Rd\,,
\end{equation*}
and employ the argument in the proof of \cref{L5.3} to show
that
\begin{equation*}
\sE_x(c,v)\,\le\, \sE_x\bigl(c + t_\circ\Ind_\sB,v\bigr)  \,=\,
\Hat\lambda(\cA^c_v+t_\circ\Ind_\sB) \,<\, \Hat\sE^*\,.
\end{equation*}
Thus $v$ cannot be optimal, and we reach a contradiction.
This completes the proof.
\end{proof}

\section{Proofs of \texorpdfstring{\cref{T2.1,T2.2,T2.3}}{}}\label{S-proofs}
This section is devoted to the proofs of \cref{T2.1,T2.2,T2.3}.
We start with a few auxiliary results which are needed in the proofs.


We begin with the Aleksandrov--Bakelman--Pucci (ABP) estimate for $\cI$.
See also \cite{MS} for more general estimates, and \cite{GM2002}
for results on elliptic integro-differential operators with regular kernels. 

\begin{theorem}\label{T6.1}
Suppose that $u\in\Sobl^{2,p}(D)\cap\cC_b(\Rd)$, $p>d$, $c\leq 0$, and a constant
$M>0$ satisfy
\begin{equation*}
\trace (a\grad^2 u)(x)+ \sup_{\zeta\in\Act} I[u,x,\zeta]+ M\abs{\grad u(x)} + c(x) u
\,\ge\, f(x)
\quad \text{in\ } \{u>0\}\cap D\,, \quad\text{with\ } u\le 0 \text{\ in\ } D^c\,.
\end{equation*}
Then for some constant $B$, which depends on  $M$,
$\diam D$, $\nu$, and $\upkappa$ in \cref{A2.1}, we have
\begin{equation*}
\sup_{D} u^+ \,\le\,  B \norm{f^-}_{L^d(D)}\,.
\end{equation*}
\end{theorem}

\begin{proof}
We write
\begin{equation*}
\trace (a\grad^2 u)(x) + 
\sup_{\zeta\in\Act}\int_{z\colon x+z \in D}\langle\grad u(x), z\rangle\,\nu(x,\zeta, \D{z})
+ M\abs{\grad u(x)}\,\ge\, f(x) -g(x)\quad \text{in\ } \{u>0\}\,,
\end{equation*}
where
\begin{multline*}
g(x) \,\df\, \sup_{\zeta\in\Act}\biggl[\int_{z\colon x+z \in D} (u(x+z)-u(x)- \grad u(x)\cdot z) 
\nu(x,\zeta,\D{z})\\
+ \int_{z\colon x+z \in D^c}(u(x+z)-u(x))\nu(x,\zeta,\D{z})\biggr]\,.
\end{multline*}
Letting $M_1= M+ \sup_{\zeta\in\Act} \int_{z\colon x+z \in D} |z| \nu(x,\zeta, \D{z})$,
 we then obtain
$$
\trace (a\grad^2 u)(x) 
+ M_1\abs{\grad u(x)}\,\ge\, f(x) -g(x)\quad \text{in\ } \{u>0\}\,.
$$
Applying \cite[Proposition~3.3]{CCKS} we obtain  
\begin{equation*}
\sup_{D}\,u^+ \,\le\, \sup_{\partial D} u^+ + B \norm{(f-g)^-}_{L^d(\Gamma^+)}
\end{equation*}
for some constant $B$,
where $\Gamma^+$ denotes the upper contact set of $u^+$ in $D$, that is, 
\begin{equation*}
\Gamma^+ \,=\, \bigl\{x \in D\,\colon\, \exists \, p\in\Rd \text{\ such that\ }
u^+(y)\le u^+(x) + p\cdot (y-x) \text{\ for\ } y\in D\bigr\}\,.
\end{equation*}
Note that for every $x\in \Gamma^+$ we have $u(x)\ge 0$ and 
\begin{equation*}
u(x+z)-u(x)- \grad u(x)\cdot z \,\le\, 0 \quad \text{for\ } x+z\in D\,.
\end{equation*}
Thus we get $g\le 0$ on $\Gamma^+$. Hence $(f-g)^-\le f^-$, and the result follows.
\end{proof}

As a consequence of \cref{T6.1} we have a narrow domain maximum principle.

\begin{theorem}\label{T6.2}
There exists $\varepsilon>0$ such that whenever $Q\subset D$, $\abs{Q}\le \varepsilon$,
and $w\in\Sobl^{2,p}(D)\cap\cC_b(\Rd)$, $p>d$, satisfy 
\begin{equation*}
\cI w \,\ge\, 0 \quad \text{in\ } D\,, \quad w\le 0 \text{\ in\ } Q^c\,,
\end{equation*}
then $w\le 0$ in $D$.
The same applies for the operator $\sA^c$.
\end{theorem}

\begin{proof}
Let $M$ be such that $\sup_{D\times\Act} \abs{b(x,\zeta)}\le M$.
Let $f=-\norm{c}_\infty\abs{w}$.
Note that on $\{w>0\}$ we have $f^-= \norm{c}_\infty w^+$. 
Then the result follows from \cref{T6.1}.
\end{proof}

Now we are ready to state the existence result.

\begin{theorem}\label{T6.3}
Suppose that $c\le 0$, and $D$ be a bounded $\cC^{1,1}$ domain in $\Rd$.
Then for any $f\in \cC(\Bar D)$, there exists a unique solution 
$u\in\cC_0(D)\cap\Sob^{2,p}(D)$ satisfying
\begin{equation*}
\cI u \,=\, f 
\quad(\text{or\ } \sA^c u\,=\,f)\quad \text{in\ } D\,, \quad u=0\quad \text{in\ } D^c\,.
\end{equation*}
\end{theorem}

\begin{proof}
For $u\in \cC_0(D)$, and define 
\begin{equation*}
\cG[M,p,r, x] \,\df\,\trace (a M)
+ \inf_{\zeta\in\Act}\,\bigl\{I[u, x, \zeta] + b(x,\zeta)\cdot  p + c(x,\zeta) r\bigr\},
\end{equation*}
for $x\in D$ and $(M, p, r)\in\mathbb{S}^d\times\Rd\times\RR$. Then note that
$G[M, p, r, x]-G[0, 0, 0, x]$ satisfies the conditions of \cite{W09}.
Thus, by \cite[Theorem~4.6]{W09}, there exists a unique solution
$v\in \cC_0(D)\cap\Sob^{2,p}(D)$, $p>d$, to 
\begin{equation*}
G[\grad^2v, Dv, v, x]\,=\, f(x)\,,
\end{equation*}
satisfying
\begin{equation}\label{PT6.3B}
\norm{v}_{\Sob^{2,p}(D)}
\,\le\, \kappa\Bigl(\norm{v}_\infty + \bnorm{f-\cG[0,0,0,\cdot\,]}_{L^p(D)}\Bigr)
\end{equation}
for some constant $\kappa=\kappa(p, D)$ which does not depend on $u$, $v$, or $f$.
Using the Aleksandrov--Bakelman--Pucci (ABP) estimate (see for example,
\cite{CCKS}, \cite[Theorem~3.1]{QS08}) we deduce that 
\begin{equation*}
\norm{v}_\infty \,\le\, \kappa_1 \bnorm{f-\cG[0,0,0,\cdot\,]}_{L^d(D)}
\end{equation*}
for some constant $\kappa_1$ which depends on $a$, $D$, and a bound of $b$.
Thus by \cref{PT6.3B} we obtain
\begin{equation}\label{PT6.3C}
\norm{v}_{\Sob^{2,p}(D)} \,\le\, \kappa_2 \bnorm{f-\cG[0,0,0,\cdot\,]}_{L^p(D)}
\end{equation}
for some constant $\kappa_2$.
Let $\cT u =v$ denote the operator mapping $u\in \cC_0(D)$ to this solution.
Since the embedding $\Sob^{2,p}(D)\hookrightarrow\cC^{0,\alpha}(D)$
is compact for $p>d$ and $\alpha\in(0,1-\nicefrac{d}{p})$,
it follows from \cref{PT6.3C} that $\cT$ is a compact operator.
From the same estimate
it is also easy to see that $u\mapsto \cT u$ is continuous in $\cC_0(D)$.
We claim that the set
\begin{equation*}
\bigl\{u\in \cC_0(D)\,\colon\, u=\mu \cT u \text{\ for some\ } \mu\in [0, 1]\bigr\}
\end{equation*}
is bounded in $\cC_0(D)$.
To prove the claim, we argue by contradiction.
If not, there must exists a sequence $(u_n, \mu_n)$ with $\norm{u_n}_\infty\to \infty$
and $\mu_n\to\mu\in[0,1]$ as $n\to\infty$.
Using \cref{PT6.3C}, scaling the solution so that $\norm{u_n}_\infty=1$,
and extracting a subsequence of
$\{u_n\}$, we obtain a nontrivial nonzero solution $w\in \cC_0(D)$ of 
\begin{equation*}
\trace (a \grad^2 w)
+ \inf_{\zeta\in\Act}\,\bigl\{\mu I[w, x, \zeta] + b(x,\zeta)\cdot \grad w
+  c(x,\zeta) w\bigr\}=0
\end{equation*}
for some $\mu\in [0,1]$.
But this contradicts the ABP maximum principle in \cref{T6.1},
thus proving the claim.
Therefore, by the Leray--Schauder fixed point theorem, there exists a fixed point
$u\in\cC_0(D)\cap\Sob^{2,p}(D)$ of $\cT$.
This proves the existence of a solution.
Uniqueness follows from the ABP estimate (\cref{T6.1}).
This completes the proof.
\end{proof}

Let us also recall the version of the nonlinear Krein--Rutman theorem
in \cite[Theorem~1]{A18}.

\begin{theorem}\label{T6.4}
Let $\sP$ be a nonempty cone in an ordered Banach space $\cX$.
Suppose that $\cT\colon\cX\to\cX$ is order-preserving, $1$-homogeneous,
completely continuous map such that
for some nonzero $u$, and $M>0$ we have $u\preceq M \cT u$.
Then there exists $\lambda>0$ and $x\ne 0$ in $\sP$ such that $\cT x=\lambda x$.
\end{theorem}

In the above theorem, `$\preceq$' denotes the partial ordering with respect to $\sP$.
Assume $c\le 0$. Let $\cX=\cC_0(D)$ and $\sP$ be the cone of nonnegative functions.
For our purposes, given $u\in\cC_0(D)$,
we let $v=\cT u\in \cC_0(D)\cap\Sob^{2,p}(D)$ denote the solution of
\begin{equation*}
\cI v(x) \,=\, -u(x)\quad \text{in\ } D\,, \quad \text{and\ } v=0 \text{\ in\ } D^c\,.
\end{equation*}
This map is well defined by \cref{T6.3}.
Since the operator is proper (i.e., it is non-increasing with respect to
the zeroth order term)
we can apply \cref{T6.1} to obtain
\begin{equation*}
\sup_{D}\,\abs{v} \,\le\, \kappa\, \sup_{D}\,\abs{u}
\end{equation*}
for some constant $\kappa$, not depending on $u$, or $v$.
Next we write 
\begin{equation*}
\cI v - \inf_{\zeta\in\Act} I[v, x, \zeta] \,=\, -u - \inf_{\zeta\in\Act} I[v, x, \zeta]\,,
\end{equation*}
and apply \cite[Theorem~4.6]{W09} to obtain
\begin{equation*}
\norm{v}_{\Sob^{2,p}(D)} \,\le\, \kappa_1 \sup_{D}\,\abs{u}
\end{equation*}
for some constant $\kappa_1$.
This of course implies that $\cT $ is an compact operator.
It is also standard to show that it is continuous. 
It is easily seen that $\cT$ is $1$-homogeneous.
Also note that $\cT(\sP)\subset\sP$.
In fact, for $u_1\le u_2$, we have
$\cI v_1(x) \ge \cI v_2(x)$.
From the concavity of $\cI$, this gives us
$\cI (v_2-v_1) \le 0$.
Thus, since $\cI$ is a proper operator, we see from \cref{T6.1} that $v_2\ge v_1$.
This inequality is strict if $u_1\lneq u_2$. 
Indeed, with $v=v_2-v_1$,  we have 
\begin{equation*}
\begin{aligned}
\trace (a\grad^2 v)(x) +\inf_{\zeta\in\Act} \{b\bigl(x,\zeta)\bigr)\cdot\grad v(x)
-  \bigl(\nu(x,\zeta,\Rd)-c\bigl(x,\zeta(x)\bigr)\bigr) v(x)\}
&\,\le\, \cI v(x) \\
&\,\le\, u_1(x)-u_2(x) \,\lneq\, 0\,.
\end{aligned}
\end{equation*}
Then by a version of Hopf's boundary lemma \cite[Lemma~3.1]{QS08},
we must have $v=v_2-v_1>0$ in $D$.

Now consider a function
$u\in \sP$ which is compactly supported in $D$, $u\ne0$.
It follows from the analysis above that $v=\cT u >0$ in $D$.
Thus we can find $M>0$ satisfying $M\cT u-u>0$ in $D$. 
Therefore, by the Krein--Rutman theorem (\cref{T6.4} above) we have $\lambda>0$
and $\psi>0$ in $D$ such that
\begin{equation*}
\cI \psi \,=\, \lambda \psi\quad \text{in\ } D\,, \quad\text{and\ \ }
\psi=0\ \ \text{on\ }D^c\,.
\end{equation*}

\begin{proof}[Proof of \cref{T2.1}]
Since $c$ is bounded, replacing $c$ by $c-\norm{c}_\infty$ it follows
from the above discussion that there 
exists $\lambda\in\RR$ and $\psi\in \cC(\Bar{D})\cap\Sobl^{2,p}(D)$, $p>d$,
satisfying
\begin{equation*}
\cI \psi \,=\, \lambda\, \psi \quad \text{in\ } D\,,
\qquad \psi >0 \quad \mbox{in\ } D\,,\qquad\text{and\ \ }
\psi \,=\,0\quad \text{in\ } D^c\,.
\end{equation*}
It is clear then that the proof is complete if we establish the following claim.

\noindent\textbf{Claim:} Suppose that
$u\in\cC_{b, +}(\Rd)\cap\Sobl^{2,d}(D)$, satisfies
$u>0$ in $D$ and
$\cI u \le  \lambda u$ in $D$.
Then $u= C\psi$ for some constant $C$.

Let $K$ be a compact subset of $D$ such that narrow domain maximum principle,
\cref{T6.2}, holds in $D\setminus K$.
Consider $w_t = t\psi-u$.
We can choose $t>0$ small enough so that $w_t\le 0$ in $K$.
Also note that 
\begin{equation*}
\trace (a\grad^2 w_t)(x) + \sup_{\zeta\in\Act}I[w_t,x, \zeta] + \sup_{\zeta\in\Act}\,
\bigl\{b(x,\zeta)\cdot \grad w_t(x)
+ (c(x,\zeta)-\lambda) w_t(x)\bigr\} \,\ge\, 0\,.
\end{equation*}
Applying \cref{T6.2} we see that $w_t\le 0$ in $D$. Since 
\begin{equation*}
\trace (a\grad^2 w_t)(x) - w_t(x) \inf_{\zeta\in\Act}\nu(x,\zeta,\Rd)
+ \sup_{\zeta\in\Act}\, \bigl\{b(x,\zeta)\cdot \grad w_t(x)
- (c(x,\zeta)-\lambda)^- w_t(x)\bigr\} \,\ge\,  0\,,
\end{equation*}
applying the strong maximum principle \cite[Theorem~9.6]{GilTru},
we must either have $w_t=0$ or $w_t<0$ in $D$.
Suppose that the second option holds. Then we may define
\begin{equation*}
\mathfrak{t} \,=\, \sup\,\{t>0 \,\colon\, w_t<0 \quad \text{in\ } D\}\,.
\end{equation*}
By the above argument, $\mathfrak{t}>0$, and by strong maximum principle
\cite[Theorem~9.6]{GilTru} we must have either $w_\mathfrak{t}=0$
or $w_\mathfrak{t}<0$.
If $w_\mathfrak{t}<0$, then for some $\delta>0$ we have $w_\mathfrak{t+\delta}<0$ in $K$,
and therefore, repeating the argument above, we obtain $w_\mathfrak{t+\delta}<0$ in $D$.
This contradicts the definition of $\mathfrak{t}$. So the only possibility
is $w_\mathfrak{t}=0$, which implies that $u=\mathfrak{t}\psi$.
This proves the  claim, and completes the proof of the theorem.
\end{proof}

\begin{proof}[Proof of \cref{C2.1}]
Suppose that $\lambda_{D'}=\lambda_D$.
Then by \cref{T2.1} we have $\cI\psi_{D'}=\lambda_{D'} \psi_{D'}$ in $D'$,
and $\psi_{D'}>0$ in $D'$.
Then it follows from the proof of \cref{T2.1} that $\psi_D=\psi_{D'}$ in $D$,
which is a contradiction as $D\subsetneq D'$.
\end{proof}

We need the following boundary estimate for the proof of \cref{T2.2}.

\begin{lemma}\label{L6.1}
Suppose that $\norm{u}_\infty\le 1$, and it satisfies
\begin{equation*}
\trace (a\grad^2 u) + \delta\abs{\grad u} \,\ge\, L \text{\ \ in\ } Q\,, \quad u=0
\text{\ \ in\ } Q^c\,,
\end{equation*}
where $Q\subset D$ is a subdomain of $D$
having the exterior sphere property with radius $r>0$.
Then for $s\in(0,1)$, there exist constants $M$, and $\varepsilon$,
depending only on
$\delta$, $L$, $r$, and $s$, such that
\begin{equation*}
\abs{u(x)}\,\le\, M \dist(x,\partial Q)^s, \quad \text{for all\ } x\in Q
\text{\ such that\ }
\dist(x, \partial Q)<\varepsilon\,.
\end{equation*}
\end{lemma}

\begin{proof}
Translating the origin if needed,
let $\sB_r$ be a ball of radius $r$ centered at $0$ that touches
$\partial Q$ from outside.
Without loss of generality we assume $\overline\sB_r\subset D$.
Define $\rho(x) = M (\abs{x}-r)^s$.
Then an easy calculation shows that we can find a constant $M>1$ satisfying
\begin{equation*}
 \trace (a\grad^2 \rho) + \delta\abs{\grad \rho} \,<\,-L \quad \text{in\ }
\sB_{r+\varepsilon}\setminus \overline{\sB}_r\,,
\end{equation*}
and $\rho\ge 1$ in $\sB^c_{r+\varepsilon}$.
The result follows from applying the comparison principle in
$(\sB_{r+\varepsilon}\setminus \overline{\sB}_r)\cap Q$.
\end{proof}

We are now ready to prove \cref{T2.2}.

\begin{proof}[Proof of \cref{T2.2}]
Let $\lim_{n\to\infty}\, \lambda_{D_n}=\Tilde\lambda$.
Note that $\Tilde\lambda \ge \lambda_D$.
Normalize the eigenfunctions to satisfy $\norm{\psi_{D_n}}_\infty=1$.
Using \cref{L6.1} and the interior estimate, it can be easily seen that the family
$\{\psi_{D_n}\}$ is equicontinuous and each limit point $u$ is a solution to
$\cI u = \Tilde\lambda u$ (or $\sA^c u = \Tilde\lambda u$).
By the strong maximum principle, we must have $u>0$.
It then follows from the
proof of \cref{T2.1} that $\Tilde\lambda=\lambda_D$.
\end{proof}

\begin{proof}[Proof of \cref{T2.3}]
We start with part (i).
It follows from the definition that $\lambda_D(c)\le\lambda_D(c')$.
Suppose that $\lambda_D(c)=\lambda_D(c')$.
Let $\psi_{c}$ and $\psi_{c'}$ denote the
principal eigenfunctions corresponding to $c$ and $c'$, respectively. Then 
\begin{equation*}
\cI \psi_{c'}(x)
\,\le\,  \lambda_D(c') \psi_{c'}(x)\quad \text{in\ }D\,,
\end{equation*}
and the proof of \cref{T2.1} shows that
the eigenfunction $\psi_c$ must be of the form $\kappa\psi_{c'}$ for some $\kappa>0$.
This contradicts the fact that $c\lneq c'$.
The exact same argument holds for $\sA^c$.

Next we prove that $\lambda_D(c)$ is a convex function of $c$.
Let $\varphi_0$ and $\varphi_1$ denote the eigenfunctions
corresponding to potentials $c_0$ and $c_1$, respectively.
Define $\varphi(x)=\varphi_0^\theta(x) \varphi_1^{1-\theta}(x)$.
Since $\varphi_0, \varphi_1>0$ in $D$, it is easy to see that
$\varphi\in\Sobl^{2, d}(D)\cap \cC(\Rd)$.
An easy calculation shows that
\begin{align*}
\trace(a\grad^2 \varphi) &\,=\, \theta \frac{\varphi}{\varphi_0}
\trace\bigl(a\grad^2 \varphi_0\bigr)
+ (1-\theta) \frac{\varphi}{\varphi_1} \trace\bigl(a\grad^2 \varphi_1\bigr)\\
&\mspace{50mu} + \varphi \left\langle\frac{\theta}{\varphi_0} \grad \varphi_0
+ \frac{1-\theta}{\varphi_1} \grad \varphi_1, a  \frac{\theta}{\varphi_0}
\grad \varphi_0+ \frac{1-\theta}{\varphi_1} \grad \varphi_1\right\rangle
\\
&\mspace{100mu}- \varphi \left(\frac{\theta}{\varphi^2_0}
\langle \grad\varphi_0, a\grad \varphi_0\rangle
+ \frac{1-\theta}{\varphi^2_1} \langle \grad\varphi_1, a\grad \varphi_1\rangle \right)
\\
&\,\le\, \theta \frac{\varphi}{\varphi_0} \trace(a\grad^2 \varphi_0)
+ (1-\theta) \frac{\varphi}{\varphi_1} \trace(a\grad^2 \varphi_1)\,,
\end{align*}
where the last line follows from convexity. Also, by Minkowski's inequality
\begin{align*}
I[\varphi,x] &\,=\, \varphi \int_{\Rd} \left(\frac{\varphi_0^{\theta}(x+z)}
{\varphi_0^\theta(x)} \frac{\varphi_1^{1-\theta}(x+z)}{\varphi_1^{1-\theta}(x)} -1\right)
\, \nu(x,\D{z})\\
&\,\le\, \varphi(x) \int_{\Rd} \left(\theta \frac{\varphi_0(x+z)}{\varphi_0(x)}
+(1-\theta) \frac{\varphi_1(x+z)}{\varphi_1(x)} -1\right)\, \nu(x,\D{z})\\
&\,=\, \theta \frac{\varphi(x)}{\varphi_0(x)} I[\varphi_0, x]
+ (1-\theta) \frac{\varphi(x)}{\varphi_1(x)} I[\varphi_1, x]\,.
\end{align*}
Thus, combining the above estimates. it follows that
with $c=\theta c_1 +(1-\theta) c_2$, we have
\begin{equation*}
\sA^c \varphi(x) \,\le\,
\bigl(\theta \lambda_D(c_0) + (1-\theta)\lambda_D(c_1)\bigr)\varphi(x)\,,
 \quad x\in D\,.
\end{equation*}
Therefore,
 $\lambda_\theta \le \theta \lambda_D(c_0) + (1-\theta)\lambda_D(c_1)$,
and the proof of part (ii) is complete.
 
The proof for the operator $\cI$ is essentially the same.
\end{proof}

\section*{Acknowledgment}
The research of Ari Arapostathis was supported
in part by the National Science Foundation through grant DMS-1715210,
in part by the Army Research Office through grant W911NF-17-1-001,
and in part by the Office of Naval Research through grant N00014-16-1-2956
and was approved for public release under DCN \#43-4933-19.
The research of Anup Biswas was supported in part by an INSPIRE faculty fellowship
and DST-SERB grants EMR/2016/004810, MTR/2018/000028.



\end{document}